\documentclass[11pt]{amsart}

\usepackage{amssymb,amsmath}
\usepackage{amsthm,enumerate}
\usepackage{graphicx,rotating}
\usepackage{qsymbols}
\usepackage{latexsym}
\usepackage{amsbsy}
\usepackage{amsfonts}
\usepackage{dsfont,txfonts}
\usepackage{color}

\newtheorem{thm}[equation]{Theorem}
\newtheorem{lem}[equation]{Lemma}
\newtheorem{prop}[equation]{Proposition}
\newtheorem{cor}[equation]{Corollary}
\newtheorem{defi}[equation]{Definition}
\newtheorem{rem}[equation]{Remark}
\newtheorem{claim}[equation]{Claim}
\numberwithin{equation}{section}

\def\club{{\huge $\mathbf{\clubsuit}$}}

\newcommand{\X}{\mathcal{X}}

\newcommand{\R}{\mathbb{R}}
\newcommand{\N}{\mathbb{N}}
\newcommand{\e}{\varepsilon}
\newcommand{\1}{\mathds{1}}
\renewcommand{\P}{\mathbb{P}}
\newcommand{\E}{\mathbb{E}}
\newcommand{\ent}{\operatorname{Ent}}

\newcommand{\Tw}{\widetilde{\mathcal{T}}}

\begin{document}
\title[Displacement convexity  on graphs]
{Displacement convexity of entropy and related inequalities on graphs}
\author{Nathael Gozlan, Cyril Roberto, Paul-Marie Samson, Prasad Tetali}

\thanks{Supported by the grants ANR 2011 BS01 007 01,  ANR 10 LABX-58, and the NSF DMS 1101447.}

\date{\today}

\address{Universit\'e Paris Est Marne la Vall\'ee - Laboratoire d'Analyse et de Math\'e\-matiques Appliqu\'ees (UMR CNRS 8050), 5 bd Descartes, 77454 Marne la Vall\'ee Cedex 2, France}
\address{Universit\'e Paris Ouest Nanterre La D\'efense - Modal'X, 200 avenue de la R\'epublique 92000 Nanterre, France}
\address{School of Mathematics \& School of Computer Science, Georgia Institute of Technology,
Atlanta, GA 30332}
\email{nathael.gozlan@univ-mlv.fr, cyril.roberto@math.cnrs.fr,  paul-\linebreak marie.samson@univ-mlv.fr, tetali@math.gatech.edu}
%\subjclass{60E15, 32F32 and 26D10}

\begin{abstract}
We introduce the notion of an interpolating path on the set of probability measures on finite graphs.  Using this notion, we first prove a displacement convexity property of entropy along such a path and derive Prekopa-Leindler type inequalities, a Talagrand transport-entropy inequality, certain HWI type as well as  log-Sobolev type inequalities in discrete settings. To illustrate through examples,  we apply our results to the complete graph and to the hypercube for which our results are optimal -- by passing to the limit, we recover the classical log-Sobolev inequality for the standard Gaussian measure with the optimal constant. 
\end{abstract}

\maketitle

\section{Introduction}

In recent years, Optimal Transport and its link with the Ricci
curvature in Riemannian geometry attracted a considerable amount of attention. The extensive
modern book by C. Villani \cite{villani} is one of the main references on this topic. However, while a
lot is now known in the Riemannian setting (and more generally in geodesic spaces), very little is
known so far in discrete spaces (such as finite graphs or finite Markov chains), with the notable
exception of some notions of (discrete) Ricci curvature proposed recently by several authors --
unfortunately there is not yet a satisfactory (universally agreed upon) resolution even there --
see  Bonciocat-Sturm \cite{bonciocat}, Erbar-Maas \cite{erbar-maas}, Hillion \cite{hillionthesis}, Joulin \cite{joulin}, Lin-Yau \cite{lin-yau},    Maas \cite{maas},  Mielke \cite{mielke}, Ollivier \cite{ollivier}, and recent works on the displacement convexity of  entropy by Hillion \cite{hillion}, Lehec\cite{lehec} and L\'eonard \cite{leonard}. 

In particular,  the notions of Transport inequalities, 
HWI inequalities, interpolating paths on the measure space, displacement convexity of entropy, are yet to be properly introduced, analyzed and understood in discrete spaces. 
This is the chief aim of the present paper, and of a companion paper \cite{GRST}.
Due to its theoretical as well as applied appeal, this subject is at the intersection of many
areas of Mathematics, such as Calculus of Variations, Probability Theory, Convex Geometry
and Analysis, as well as Combinatorial Optimization.

\medskip

In order to present our results, let us first introduce some of the relevant notions in the continuous framework of geodesic spaces, see \cite{villani}.

A  complete, separable, metric space $(\X,d)$ is said to be a \emph{geodesic space}, if for all $x_0,x_1 \in \X$, there exists at least one path $\gamma \colon [0,1] \mapsto \X$ such that $\gamma(0)=x_0, \gamma(1)=x_1$ and 
$$
d(\gamma(s),\gamma(t))=|t-s|d(x_0,x_1), \qquad \forall s,t \in [0,1].
$$
Such a path is then called a \emph{constant speed geodesic} between $x_0$ and $x_1$.

Then, for $p \geq 1$, let $\mathcal{P}_p(\X)$ be the set of Borel probability measures on $\X$ having a finite $p$-th moment, namely
$$
\mathcal{P}_p(\X) := \left\{ \mu \mbox{ Borel probability measure}: \int_\X d(x_o,x)^p \mu(dx) < + \infty  \right\} \,,
$$
where $x_o \in \X$ is arbitrary ($\mathcal{P}_p(\X)$ does not depend on the choice of the point $x_o$)
and define the following $L_p$-Wasserstein distance: for $\nu_0,\nu_1 \in \mathcal{P}_p(\X)$, set
\begin{equation} \label{t2}
W_p(\nu_0,\nu_1):= \left( \inf_{\pi \in \Pi(\nu_0,\nu_1)} \left\{ \iint d(x,y)^p \,d\pi(x,y) \right\} \right)^{1/p}\,,
\end{equation}
where $\Pi(\nu_0,\nu_1)$ is the set of couplings of  $\nu_0$ and $\nu_1$. 

The metric space $({\mathcal P}_{p}(\X), W_p)$ is canonically associated to the original metric space $(\X,d)$.
Namely, if $p>1$, $({\mathcal P}_{p}(\X), W_p)$ is geodesic if and only if $(\X,d)$ is geodesic, see \cite{St06a}.

A remarkable and powerful fact is that, when $\X$ is a Riemannian manifold, one can relate the Ricci curvature of the space to the convexity of  entropy along geodesics \cite{mccann,cordero-ms,sturm-vonrenesse,LV09,sturm05,villani03}. 
 More precisely, under the Bakry-Emery ${\rm CD}(K, \infty)$ condition (see e.g. \cite{bakry}), namely if the space 
$(\X,d,\mu)$ is such that $\mathrm{Ric}+\mathrm{Hess}\,V\geq K$, where $\mu(dx)=e^{-V(x)}\,dx$, then one can prove that for all $\nu_0,\nu_1\in\mathcal{P}_{2}(\X)$ whose supports are included in the support of $\mu$, there exists a constant speed ${W}_{2}$-geodesic
$\{\nu_t\}_{t\in [0,1]}$ from $\nu_0$ to $\nu_1$  such that
\begin{equation} \label{dc}
H(\nu_t|\mu) \leq (1-t) H(\nu_0|\mu) + t H(\nu_1|\mu) - \frac{K}{2}t(1-t) {W}_2^2(\nu_0,\nu_1)  \qquad \forall  t\in [0,1], 
\end{equation}
where $H(\nu|\mu)$ denotes the relative entropy of $\nu$ with respect to  $\mu$. 
Equation \eqref{dc} is known as the \emph{$K$-displacement convexity of the entropy}.
In fact, a converse statement also holds: if the entropy is $K$-displacement convex, then the Ricci curvature is bounded below by $K$.  This equivalence was used as a guideline for the definition of the notion of curvature in geodesic spaces by Sturm-Lott-Villani in their celebrated works \cite{LV09,St06a,St06b}. 

Moreover,  it is known that the $K$-displacement convexity of the entropy is a very strong notion that implies many well-known
inequalities in Convex Geometry and in Probability Theory, such as the Brunn-Minkowski inequality, the Prekopa-Leindler inequality, Talagrand's transport-entropy inequality, HWI inequality, log-Sobolev inequality etc., see \cite{villani}.

\bigskip

The question one would like to address is whether one can extend the above theory to discrete settings such as finite graphs, equipped with a set of probability measures on the vertices and with a natural graph distance. 

Let us mention two main obstructions. Firstly, ${W}_{2}$-geodesics do not exist in discrete settings
(the reader can  verify this fact by considering two nearest neighbors $x,y$ in the graph
%%PT: changed \X to G be consistent 
%%
%% $\X$ 
$G=(V,E)$ and  constructing a constant speed geodesic between the two Dirac measures $\delta_x, \delta_y$ at the vertices $x$ and $y$).
On the other hand, the following Talagrand's transport-entropy inequality
\begin{equation} \label{tal}
W_2^2(\nu_0,\mu) \leq C\  H(\nu_0|\mu)\,, \qquad \forall \nu_0 \in \mathcal{P}_2(V)\,
\end{equation}
(for a suitable constant $C>0$) does not hold in discrete settings unless $\mu$ is a Dirac measure!
From these simple observations we deduce that ${W}_{2}$ is not well adapted either for defining the path 
$\{\nu_t\}_{t\in [0,1]}$ or for measuring the  defect/excess  in the convexity of  entropy in a discrete context.

In this paper, our contribution is to introduce the notion of an interpolating path $\{\nu_t\}_{t\in [0,1]}$ 
and of a {\em weak transport cost} $\widetilde{T}_{2}$ (that in a sense goes back to Marton \cite{marton96, marton97}
% \textcolor{red}{Am I right with the citations?}
).  These will in turn help us derive the desired displacement convexity results on finite graphs.

\bigskip

Before presenting our results, we give a brief state of the art of the field (to the best of our knowledge).

In \cite{ollivier-villani}, Ollivier and Villani prove that, on the hypercube $\Omega_n=\{0,1\}^n$, 
for any probability measures $\nu_0, \nu_1$, there exists a probability measure $\nu_{1/2}$ (concentrated on the set of 
mid-points, see \cite{ollivier-villani} for a precise definition) such that
$$
H(\nu_{1/2}|\mu) \leq \frac{1}{2} H(\nu_0|\mu) + \frac{1}{2} H(\nu_1|\mu) - \frac{1}{80n} W_1^2(\nu_0,\nu_1)\,,
$$
where $\mu\equiv 1/2^n$ is the uniform measure and $W_1$ is defined with the Hamming distance. They observe that, this in turn implies some curved Brunn-Minkowski inequality on $\Omega_n$. The constant $1/n$ encodes, in some sense, the discrete Ricci curvature of the hypercube
in accordance with the various definitions of the discrete Ricci curvature (see above for references).

In \cite{erbar-maas}, Erbar and Maas introduce a {\em pseudo} Wasserstein distance $\mathcal{W}_2$ that corresponds to the geodesic distance on the set, $\mathcal{P}(\Omega_n)$, of probability measures on the hypercube $\Omega_n$,  equipped with a Riemannian metric. (In fact, their construction is more general and applies to a wide class of Markov kernels on finite graphs.) This metric is such that the continuous time random walk on the graph becomes a gradient flow of the function $H(\cdot|\mu)$. Moreover they prove, inter alia,  that if $\{\nu_t\}_{t\in [0,1]}$ is a geodesic from $\nu_0$ to $\nu_1$, then
$$
H(\nu_t|\mu) \leq (1-t) H(\nu_0|\mu) + t H(\nu_1|\mu) - \frac{1}{n}t(1-t) \mathcal{W}_2^2(\nu_0,\nu_1) \,,
\qquad \forall t\in[0,1] \,,
$$ 
where $\mu\equiv 1/2^n$ is the uniform measure. Independently, Mielke \cite{mielke} also obtains  similar results.
As a consequence of their displacement convexity property, these authors derive versions of log-Sobolev, HWI and Talagrand's transport-entropy inequalities (involving $\mathcal{W}_2$ and $W_1$ distances) with sharp constants.

In a different direction (at the level of functional inequalities), besides the study of the log-Sobolev inequality which is somehow now classical (see \textit{e.g.}\ \cite{Saloff,ane}),  Sammer and the last named author \cite{sammer-thesis,sammer-tetali} studied Talagrand's inequality in discrete spaces, with $W_1$ on the left hand side of \eqref{tal}. They also derived a discrete analogue of the Otto-Villani result \cite{OV00}: that
a modified log-Sobolev inequality implies the  $W_1$-type Talagrand inequality. 
Connected to this, a few years ago, 
following seminal work of Bobkov and Ledoux \cite{bobkov-ledoux-98}, several researchers
independently realized that modified versions of logarithmic Sobolev inequalities helped capture
refined information that was lost while working with the classic log-Sobolev inequality of 
Gross. In the discrete setting of finite Markov chains, one such modified log-Sobolev inequality
has been instrumental in capturing the rate of convergence to equilibrium in the (relative)
entropy sense, see e.g. \cite{PPP}, \cite{DPPP}, \cite{bobkov-tetali}, \cite{GQ03}, \cite{goel}, \cite{Saloff}, \cite{roberto-thesis}. 
The current state of knowledge in identifying precise sufficient criteria to derive bounds on the
entropy decay (or on the corresponding modified log-Sobolev constants) is unfortunately rather meagre. 
This is an independent motivation for our efforts  at
developing the discrete aspects of the displacement convexity property and related notions.

\bigskip
Now we describe some of the main results of the present paper.
At first, we shall introduce the notion of an interpolating path $\{\nu_t^\pi\}_{t\in [0,1]}$, on the set of probability measures on graphs, between two arbitrary probability measures $\nu_0, \nu_1$. In fact,  we define a {\em family} of interpolating paths, depending on a parameter $\pi \in \Pi(\nu_0,\nu_1)$, which is a coupling of $\nu_0, \nu_1$. The construction of this interpolating path is inspired by a  certain binomial interpolation due to  Johnson \cite{johnson}, see also \cite{hillionthesis,hillion,hillion-johnson}.
In particular, we shall prove that such an interpolating path, for a properly chosen coupling $\pi^*$ -- namely an optimal coupling for $W_1$ -- is actually a $W_1$ constant speed geodesic: \textit{i.e.}\ $W_1(\nu_t^{\pi^*},\nu_s^{\pi^*})=|t-s|W_1(\nu_0,\nu_1)$ for all $s,t \in [0,1]$, with $W_1$ defined with the graph distance $d$ (see Proposition \ref{prop:georgia} below). 
Such a family enjoys a tensorisation (see Lemma \ref{tensor}) that is crucial in our derivation of the displacement convexity property on product of graphs.

Indeed, we shall prove the following tensoring property of a displacement convexity of entropy along the interpolating path $\{\nu_t^\pi\}_{t\in [0,1]}$. This is one of our main results (see below and  Theorem~\ref{th:main}). In order to state the result, we define here the notion of a quadratic cost, which we will elaborate on, in the later sections.

%%PT: moved the definition of I_2(\pi) to here
%%

%For later purposes, it will be convenient to introduce the following notation: 
Let $G=(V,E)$ be a (finite) connected, undirected graph, and let $\mathcal{P}(V)$ denote the set of probability measures on the vertex set $V$.  Given two probability measures $\nu_0$ and $\nu_1$ on $V$, let $\Pi(\nu_0, \nu_1)$ denote the set of couplings (joint distributions) of $\nu_0$ and $\nu_1$. 
%Furthermore, let $P(\nu_0, \nu_1)$ denote the set of probability kernels $\nu$ so that ...
Given $\pi \in \Pi(\nu_0,\nu_1)$, consider the probability kernels 
%$p\in P(\nu_{0},\nu_{1})$ and $\bar{p}\in P(\nu_{1},\nu_{0})$ 
$p$ and $\bar{p}$ 
defined by $$\pi(x,y)=\nu_0(x)p(x,y) = \nu_{1}(y)\bar{p}(y,x), \ \ \forall x,y \in V, $$ and set
\begin{eqnarray}%\label{weakbis}
I_2(\pi):=  \sum_{x \in V} \left( \sum_{y \in V} d(x,y) p(x,y) \right)^2 \nu_0(x),
\end{eqnarray}
\begin{eqnarray*}
\bar{I}_2(\pi):=  \sum_{y \in V} \left( \sum_{x \in V} d(x,y) \bar{p}(y,x) \right)^2 \nu_1(y)\,.
\end{eqnarray*}
%and
%\begin{eqnarray*}
%J_2(\pi):=   \left(\sum_{x,y \in V} \sum_{x \in V} d(x,y)\pi(x,y) \right)^2.
%\end{eqnarray*}
%With this notation, 
%$$\Tw_{2}(\nu_{0}|\nu_{1})=\inf_{\pi \in \Pi(\nu_{0},\nu_{1})}I_{2}(\pi) .
%$$

We say a graph $G$, equipped with the distance $d$ and probability measure $\mu\in \mathcal{P}(V)$, satisfies the displacement convexity property (of entropy), if there exists a $C = C(G, d, \mu) > 0$, so that for any $\nu_0, \nu_1 \in \mathcal{P}(V)$,  there exists a $\pi\in \Pi(\nu_0, \nu_1)$ satisfying:
$$
H(\nu_t^\pi | \mu) \leq (1-t)H(\nu_0 | \mu) + tH(\nu_1 | \mu) - C t(1-t)(I_2(\pi)+\bar{I}_2(\pi))\,, \qquad \forall t\in [0,1].
$$

The quantity $I_2(\pi)$ goes back to Marton \cite{marton96,marton97}
% \textcolor{red}{to be cited properly}
in her definition of the following transport cost, we call \emph{weak transport cost}:
$$
\widetilde W_2^2(\nu_0,\nu_1) := \inf_{\pi \in \Pi(\nu_0,\nu_1)}  I_2(\pi) + \inf_{\pi \in \Pi(\nu_0,\nu_1)}  \bar{I}_2(\pi)\,.
$$
For more on this Wasserstein-type distance, see \cite{dembo97, marton99, S00}.  The precise statement of our tensorisation theorem is as follows. For  a graph, by the {\em graph distance} between two vertices, we mean the length of a shortest path between the two vertices.

\begin{thm} \label{th:mainintro}
For $i\in \{1,\ldots,n\}$, let $\mu^i$ be a probability measure on  $G_i=(V_i, E_i)$, with the graph distance $d_i$.  Assume also that for each $i\in \{1,\ldots,n\}$ there is a constant $C_i\geq 0$ such that for all probability measures $\nu_0,\nu_1$ on $V_i$, there exists $\pi = \pi^i \in \Pi(\nu_0,\nu_1)$ such that  it holds
$$
H(\nu_t^\pi | \mu^i) \leq (1-t)H(\nu_0 | \mu^{i}) + tH(\nu_1 | \mu^{i}) - C_it(1-t)(I_2(\pi)+\bar{I}_2(\pi)) \qquad \forall t\in [0,1].
$$
Then the product probability measure $\mu=\mu^1\otimes\cdots\otimes\mu^n$ defined on the Cartesian product $G=G_1\boxempty\cdots\boxempty G_n$ (see below for a precise definition) verifies the following property: 
for all probability measures $\nu_0,\nu_1$ on $V$, there exists $\pi = \pi^{(n)}\in \Pi(\nu_0,\nu_1)$ satisfying,
$$
H(\nu_t^\pi | \mu) \leq (1-t)H(\nu_0 | \mu) + tH(\nu_1 | \mu) - Ct(1-t)(I_2^{(n)}(\pi)+\bar{I}^{(n)}_2(\pi))
\qquad \forall t\in [0,1] ,
$$
where $C=\min_i C_i$,
$$
I_2^{(n)}(\pi) := \sum_{x \in V_1 \times \cdots \times V_{n}} \sum_{i=1}^n  \left( \sum_{y \in V_1 \times\cdots \times V_{n}} d_i(x_i,y_i) \frac{\pi(x,y)}{\nu_0(x)} \right)^2 \nu_0(x) ,
$$
and
$$
\bar I_2^{(n)}(\pi) := \sum_{y \in V_1 \times \cdots \times V_{n}} \sum_{i=1}^n  \left( \sum_{x \in V_1 \times\cdots \times V_{n}} d_i(x_i,y_i) \frac{\pi(x,y)}{\nu_1(y)} \right)^2 \nu_1(y) .
$$
(and with $I_2(\pi):=I_2^{(1)}(\pi)$ and similarly for $\bar I_2(\pi)$).
\end{thm}

In particular, as a consequence of the above tensorisation theorem, we shall prove that, given two probability measures $\nu_0, \nu_1$ on the hypercube $\Omega_n=\{0,1\}^n$, there exists a coupling $\pi$ such that
\begin{equation} \label{dcintro}
H(\nu_t^\pi|\mu) \leq (1-t) H(\nu_0|\mu) + t H(\nu_1|\mu) - \frac{1}{2}t(1-t) \widetilde W_2^2(\nu_0,\nu_1) \,,
\qquad \forall t\in[0,1] 
\end{equation}
where $\mu \equiv 1/2^n$ is the uniform measure (but that could be any product of Bernoulli measures).
As it is easy to see, the weak transport cost is weaker than $W_2$, but stronger than $W_1$. Moreover,
$\widetilde W_2^2(\nu_0,\nu_1) \geq \frac{2}{n} W_1^2(\nu_0,\nu_1)$ (see below) so that \eqref{dcintro} captures, in a sense, a discrete Ricci curvature of the hypercube (see \cite{ollivier-villani} and references therein).

As a by-product of the displacement convexity property above, we shall derive a series of consequences. More precisely, we shall first derive a so-called HWI inequality.

\begin{prop}\label{HWIpropintro}
Let  $\mu$ be a probability measure on  $V^n$. 
Assume that $\mu$ verifies the following displacement convexity inequality: there is some $c>0$ such that for any probability measures $\nu_0, \nu_1$ on $V^n$, there exists a coupling $\pi \in \Pi(\nu_0,\nu_1)$ such that
$$
H(\nu_t^\pi |\mu) \leq (1-t) H(\nu_0|\mu) + t H(\nu_1|\mu) -ct(1-t) (I_2^{(n)} (\pi)+\bar{I}_2^{(n)}(\pi)) \qquad \forall t \in [0,1].
$$
Then $\mu$ verifies 
\begin{align*}
H(\nu_0|\mu) 
&\leq 
H(\nu_1|\mu) + \sqrt{\sum_{x\in V^n} \sum_{i=1}^n \left[\sum_{z\in N_i(x)} \left( \log  \frac{\nu_0(x)}{\mu(x)} - \log \frac{\nu_0(z)}{\mu(z)}  \right)\right]_{+}^2\nu_0(x)}\sqrt{I_2^{(n)}(\pi)} %\nonumber \\
%& \quad 
- c(I_2^{(n)}(\pi) + \bar{I}_2^{(n)}(\pi)),
\end{align*}
for the same $\pi\in \Pi(\nu_0,\nu_1)$ as above, where $N_i(x)$ is the set of neighbors of $x$ in the $i$-th direction (see 
Proposition \ref{HWIprop} for a precise definition).
\end{prop}
On the hypercube, the latter implies the following log-Sobolev-type inequality (that can be seen as a reinforcement of a discrete modified log-Sobolev inequality (see Corollary \ref{cortarte})): if $\mu\equiv 1/2^n$, for any $f \colon \Omega_n \to (0,\infty)$, it holds
\begin{eqnarray*}
\ent_{\mu}(f)\leq \frac{1}{2} \sum_{x\in \Omega_n} \sum_{i=1}^n  \left[ \log  f(x) - \log f(\sigma_i(x))  \right]_{+}^2f(x)\mu(x) -  \frac{1}{2} \widetilde W_2^2 (f\mu|\mu),
\end{eqnarray*}
where $\sigma_i(x)=(x_1,\dots,x_{i-1},1-x_i,x_{i+1},\dots,x_n)$ is the vector $x=(x_1,\dots,x_n)$ with the $i$-th coordinate flipped, and the constant $1/2$ (in front of the Dirichlet form) is optimal.

From this, by means of the Central Limit Theorem, the above reinforced modified log-Sobolev inequality actually leads to
the usual logarithmic Sobolev inequality of Gross \cite{gross} for the standard Gaussian, with the optimal constant (see Corollary \ref{cor:lsob}).

In a different direction, we also prove that the displacement convexity along the interpolating path $\{\nu_t^\pi\}_{t\in [0,1]}$ implies a discrete Prekopa-Leindler Inequality (Theorem \ref{th:dpl}), which in turn, as in the continuous setting, implies a logarithmic Sobolev inequality and a (weak) transport-entropy inequality of the Talagrand-type:
$$
\widetilde W_2^2 (\nu|\mu) \leq C\  H(\nu|\mu)\,,\qquad \forall \nu\,
$$
for a suitable constant $C>0$.
These implications and inequalities are  studied in further detail -- their various links with the concentration of measure phenomenon and with other functional inequalities -- in the companion paper \cite{GRST}.

\medskip

We may summarize  the various implications that we prove in the following diagram:

$$
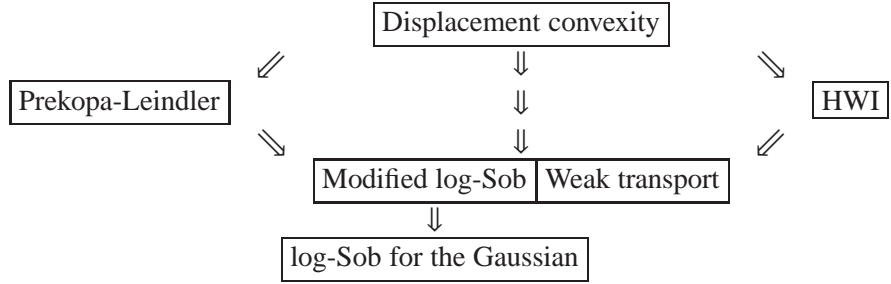
\begin{array}{ccccc}
 & & \fbox{Displacement convexity} & & \\
 & \Swarrow & \Downarrow & \Searrow & \\
 \fbox{Prekopa-Leindler} & &\Downarrow & & \fbox{HWI} \\ 
  & \Searrow & \Downarrow & \Swarrow & \\
 & & \fbox{Modified log-Sob} \fbox{Weak transport} & & \\
 &  & \hspace{-2,3cm} \Downarrow &  & \\
  & & \hspace{-2,3cm} \fbox{log-Sob for the Gaussian}& & 
\end{array}
$$

\medskip

In summary, our paper develops  various theoretical objects of much current interest (the interpolating path $\{\nu_t^\pi\}_{t\in [0,1]}$, the weak transport cost 
$\widetilde W_2$, the displacement convexity property and its consequences) in a {\em discrete} context.  
%we mainly apply to two examples of finite graphs:
Our concrete examples include the complete graph and the hypercube. However, our theory applies to other graphs (not necessarily product type) that we will collect in a forthcoming paper. 
Also, we believe that our results open a wide class of new problems and new directions of investigation in Probability Theory, Convex Geometry and Analysis.

\medskip

Finally, we mention that, during the final preparation of this work, we learned that Erwan Hillion independently introduced the same kind of interpolating path, but between a Dirac at a fixed point $o \in G$ of the graph and any arbitrary measure (hence without coupling $\pi$), and  derive some displacement convexity property \cite{hillion} along the interpolation. In \cite{hillion}, the author also deals with the $f \cdot g$ decomposition introduced by  L\'eonard \cite{leonard}.

\medskip

Our presentation follows the following table of contents.

\tableofcontents

\subsection{Notation} \label{sec:notation}

Throughout  the paper we shall use the following notation. 

\subsection*{Graphs}
$G=(V,E)$ will denote a finite connected undirected graph  with the vertex set $V$ and the edge set $E$.
For any two vertices $x$  and $y$ of $G$, $x \sim y$ means that 
$x$ and $y$ are nearest neighbors (for the graph structure of $G$), \textit{i.e.}\ $(x,y) \in E$. 
We use $d$ for the graph distance defined below.
%\textcolor{red}{\club Cyril: graphs are finite for simplicity. For infinite graphs we have to be careful about the definition of $\Tw_2$: there are some integrability issue to be solved (one should work then with the set $\mathcal{P}_2(V)$ of probability measures with finite second moment: $\int d(x,y)^2d\mu(y) <\infty$ for some $x$, but this introduce some more technicalities, in particular at the level of Prekopa). Since our examples are all finite, I suggest to restrict this paper to finite graphs, and to deal with infinite graph in the paper about Talagrand's transport-entropy inequalities.}

Given two graphs $G_1=(V_1,E_1)$, $G_2=(V_2,E_2)$, with graph distances $d_1$, $d_2$ respectively, we set 
$G_1 \boxempty G_2 = (V_1 \times V_2, E_1 \boxempty E_2)$ for the Cartesian product of the two graphs, 
equipped with the $\ell^1$ distance $d(x,y)=d_1(x_1,y_1)+d_2(x_2,y_2)$, for all
$x=(x_1,x_2), y=(y_1,y_2) \in G_1 \times G_2$. More precisely, $((x_1,x_2),(y_1,y_2)) \in E_1 \boxempty E_2$
if either $x_1=y_1$ and $x_2 \sim y_2$, or $x_1 \sim y_1$ and $x_2=y_2$.
The Cartesian product of $G$ with itself will  simply be denoted by $G^2$,
and more generally by $G^n$, for all $n\geq 2.$

%In many places, and when there is no confusion, with a slight abuse of notation, we will identify $G$ and $V$.

\subsection*{Paths and geodesics}
A \emph{path} $\gamma=(x_0,x_1,\dots,x_n)$ (of $G$) is an oriented sequence  
of vertices of $G$ satisfying $x_{i-1} \sim x_i$ for any $i=1\dots,n$. Such a path starts at $x_0$ and ends at $x_n$ and is said to be of length $|\gamma|=n$. 
The graph distance $d(x,y)$ between two vertices $x,y \in G$ is the minimal length of a path connecting $x$ to $y$. Any path of length $n=d(x,y)$ between $x$ and $y$ is called a \emph{geodesic} between $x$ and $y$. By construction, any geodesic is self-avoiding. We will denote by $\Gamma(x,y)$ the set of all geodesics from $x$ to $y$.

We will say that a path $\gamma=(x_0,x_1,\ldots,x_n)$ crosses the vertex $z\in V$, if there is some $k$ such that $z=x_k$. In this case, we will write $z\in \gamma.$ Given $z \in V$, we set $C(z)=\{(x,y) \mbox{ such that } z \in \gamma \mbox{ for some } \gamma \in \Gamma(x,y)\}$
for the set of couples such that some geodesic joining them goes through $z$.
Conversely, if $z$ belongs to some geodesic between $x$ and $y$, we shall write $z\in \llbracket x,y\rrbracket$ and say that $z$ is \emph{between} $x$ and $y$. Finally, for all $x,y,z\in V$, we will denote by $\Gamma(x,z,y)$, the set of geodesics $\gamma\in \Gamma(x,y)$ such that $z\in \gamma$. This set is nonempty if and only if $z\in \llbracket x,y\rrbracket$.

\subsection*{Probability measures and couplings}

We write ${\mathcal P}(V)$ for the set of probability measures on $V$. Given a probability measure $\nu \in {\mathcal P}(V)$ and a function $f \colon V \to \R $,  $\nu(f)=\sum_{z \in V} \nu(z)f(z)$ denotes the mean value of $f$ with respect to $\nu$. We may also use the alternative notation $\nu(f)=\int f(x) \,\nu(dx) = \int f(x)\,d\nu(x)=\int f\,d\nu$. 

Let $\nu,\mu\in \mathcal{P}(V)$; the \emph{relative entropy} of $\nu$ with respect to $\mu$ is defined by
$$
H(\nu|\mu)= 
\begin{cases}
\int \frac{d\nu}{d\mu} \log \frac{d\nu}{d\mu} \,d\mu& \mbox{if } \nu \ll \mu \\
+\infty & \mbox{otherwise} 
\end{cases}
$$
where $\nu \ll \mu$ means that $\nu$ is absolutely continuous with respect to $\mu$, and $\frac{d\nu}{d\mu}$
denotes the density of $\nu$ with respect to $\mu$.

Given a density $f \colon V \to (0,\infty)$ with respect to a given probability measure $\mu$ (\textit{i.e.}\ $\mu(f)=1$),
we shall use the following notation for the relative entropy of $f\mu$ with respect to $\mu$:
$$
\ent_\mu(f) := H(f\mu |\mu) = \int f \log f d\mu .
$$
If $f\colon V \to (0,\infty)$  is no longer a density, then  $\ent_\mu(f) := \int f \log (f/\mu(f))\,d\mu$.

Given two graphs $G_1=(V_1,E_1)$ and $G_2=(V_2,E_2)$ and a probability measure $\mu \in \mathcal{P}(V_1 \times V_2)$ on the product,
we \emph{disintegrate} $\mu$ as follows: let $\mu^2$ be the second marginal of $\mu$, \textit{i.e.}\ $\mu^2(x_2)=\sum_{x_1 \in V_1} \mu(x_1,x_2)=\mu(V_1,x_2)$, for all $x_2 \in V_2$, and set $\mu^1(x_1|x_2)$ so that
\begin{equation} \label{disintegration}
\mu(x_1,x_2)=\mu^2(x_2) \mu^1(x_1 |x_2), \qquad \forall (x_1,x_2) \in V_1 \times V_2,
\end{equation}
with the convention that $\mu^1(x_1|x_2)=0$ if $\mu^2(x_2)=0$. Equation \eqref{disintegration} will be referred to as the \emph{disintegration formula} of $\mu$.

Recall that a coupling $\pi$ of two probability measures $\mu$ and $\nu$ in $\mathcal{P}(V)$ is a probability measure on $V^2$ so that $\mu$ and $\nu$ are its first and second marginals, respectively: \textit{i.e.}\ $\pi(x,V)=\mu(x)$ and $\pi(V,y)=\nu(y)$, for all $x, y \in V$.
Given $\mu, \nu \in \mathcal{P}(V)$, the set of all couplings of $\mu$ and $\nu$ will be denoted by $\Pi(\mu,\nu)$.

Moreover, given two probability measures $\mu$ and $\nu$ in $\mathcal{P}(V)$, we denote by $P(\mu,\nu)$ the set of probability kernels\footnote{We recall that $p:V \times V \to [0,1]$ is a probability kernel if, for all 
$x \in V$, $\sum_{y \in V} p(x,y)=1$.}  $p$ such that
$$
\sum_{x \in V} \mu(x)p(x,y) = \nu(y)\,, \qquad \forall y \in V .
$$
By construction, given $p \in P(\mu,\nu)$, one defines a coupling $\pi \in \Pi(\mu,\nu)$ by setting $\pi(x,y)=\mu(x) p(x,y)$, $x,y \in V$.
Conversely, given a coupling $\pi \in \Pi(\mu,\nu)$, we canonically construct a kernel $p \in P(\mu,\nu)$
by setting $p(x,y)=\pi(x,y)/\mu(x)$ when $\mu(x) \neq 0$ and $p(x,y)=0$ otherwise.

{\it Warning 1:} In the sequel, it will always be understood, although not explicitly stated, that $p(x,y)=0$ if $\mu(x)=0$ and similarly in the disintegration formula \eqref{disintegration}.

{\it Warning 2:} For convenience, we will use the French notation $C_n^k := \genfrac{(}{)}{0pt}{}{n}{k} = \frac{n!}{k!(n-k)!}$ for the binomial coefficients.

%%%%%%%%%%%%%%%%%%%%%%%%%%%%%%%%%%%%%%%%%%%%%%%%%%%%%%%%%%%%%%%%%%%%%%%%%%%%%%%%%%%%%%%%

\section{A notion of a path on the set of probability measures on graphs.}

The aim of this section is to define a class of paths between probability measures on graphs. As proved below, each path in this class is a geodesic,
in the space of probability measures equipped with the Wasserstein distance $W_1$ (see below). It satisfies a convenient differentiation property and also has the nice feature of allowing tensorisation. 
We shall end the section with some specific examples.

%%%%%%%%%%%%%%%%%%%%%%%%%%%%%%%%%%%%%%%%%%%%%%%%%%%%%%%%%%%%%%%%%%%%%%%%%%%%%%%%%%%%%%%%

\subsection{Construction}\label{construction}

Inspired by \cite{johnson}, we will first construct an interpolating path between two Dirac measures 
$\delta_x$ and $\delta_y$,  for arbitrary  $x,y \in V$,  on the set of probability measures ${\mathcal P}(V)$. 
Fix $x,y \in V$ and denote by $\Gamma$ the random variable that chooses uniformly at random a geodesic $\gamma$ in $\Gamma(x,y)$. Also, for any $t \in [0,1]$, let $N_t \sim \mathcal{B}(d(x,y),t)$ be a binomial variable of parameter $d(x,y)$ and $t$, independent of $\Gamma$ (observe that $N_0=0$ and $N_1=d(x,y)$). Then denote by $X_t=\Gamma_{N_t}$ the random position on $\Gamma$ after $N_t$ jumps starting from $x$. Finally, set $\nu_t^{x,y}$ for the law of $X_t$.

By construction, $\nu_t^{x,y}$ is clearly a path from $\delta_x$ to $\delta_y$.
Moreover, for all $z\in V$, we have
\begin{align*}
\nu_t^{x,y}(z)&=\sum_{\gamma\in \Gamma(x,y)} \P(X_t=z|\Gamma=\gamma,z\in \Gamma) \P(\Gamma=\gamma, z\in \gamma)
= \sum_{\gamma\in \Gamma(x,y)} C_{d(x,y)}^{d(x,z)} t^{d(x,z)} (1-t)^{d(y,z)} \frac{ \1_{z\in \gamma}} {|\Gamma(x,y)|} .
\end{align*}
Therefore 
$$
\nu_t^{x,y}(z)= C_{d(x,y)}^{d(x,z)} t^{d(x,z)} (1-t)^{d(y,z)}\; \frac{|\Gamma(x,z,y)|}{|\Gamma(x,y)|}.
$$
For all $z$ between $x$ and $y$ we observe that 
\begin{equation} \label{sophiacardinale}
|\Gamma(x,z,y)|=|\Gamma(x,z)| \times |\Gamma(z,y)|,
\end{equation}
since there is a one to one correspondence between the sets
of geodesics from $x$ to $z$ and from $z$ to $y$, and the set of geodesics from $x$ to $y$ that cross the vertex $z$, just by gluing the path from $x$ to $z$ to the path from $z$ to $y$, and by using that $d(x,y)=d(x,z)+d(z,y)$.
Therefore $\nu_t^{x,y}$ takes the form
\begin{equation} \label{atlanta}
\nu_t^{x,y}(z)= C_{d(x,y)}^{d(x,z)} t^{d(x,z)} (1-t)^{d(y,z)}\; \frac{|\Gamma(x,z)| \times |\Gamma(z,y)|}{|\Gamma(x,y)|} \1_{z\in \llbracket x,y\rrbracket} .
\end{equation}
Observe that, for any $x,y \in V$ and any $t \in (0,1)$, $\nu_t^{x,y} = \nu_{1-t}^{y,x}$.

\smallskip

%%PT - added the following remark
%%
\begin{rem}
In the construction above of the interpolation $\nu^{x,y}_t$, the choice of the binomial random variable for the number $N_t$ of jumps might seem somewhat ad hoc; however, in Proposition~\ref{loident} below, we show that in fact the choice is {\em necessary} for  $\nu^{x,y}_t$ to tensorise over a (Cartesian) product of graphs.
\end{rem}
Given the family $\{\nu_t^{x,y}\}_{x,y}$, we can now construct a path from any measure $\nu_0 \in \mathcal{P}(V)$ to any measure 
$\nu_1\in \mathcal{P}(V)$. Namely, given
a coupling $\pi \in \mathcal{P}(V \times V)$ of $\nu_0$ and $\nu_1$, we define
\begin{equation} \label{atlantabis}
\nu_t^\pi(\,\cdot\,)= \sum_{(x,y) \in V^2} \pi(x,y)\nu_t^{x,y}(\,\cdot\,), \qquad \forall t \in [0,1].
\end{equation}
%where we recall that $C(z)=\{(x,y) \mbox{ such that } z \in \gamma \mbox{ for some } \gamma \in \Gamma(x,y)\}$.

By construction we have $\nu_0^\pi=\nu_0$ and $\nu_1^\pi=\nu_1$.
Furthermore, observe that, if $\nu_0=\delta_x$ and $\nu_1=\delta_y$, then necessarily $\pi = \delta_x \otimes \delta_y$ and
thus $\nu_t^\pi=\nu_t^{x,y}$.

%%%%%%%%%%%%%%%%%%%%%%%%%%%%%%%%%%%%%%%%%%%%%%%%%%%%%%%%%%%%%%%%%%%%%%%%%%%%%%%%%%%%%%%%

\subsection{Geodesics for $W_1$}

Next we prove that, when $\pi$ is well chosen, $(\nu_t^\pi)_{t \in [0,1]}$ is a geodesic from $\nu_0$ to $\nu_1$ on the set of probability measures 
$\mathcal{P}(V)$ equipped with the Wasserstein $L_{1}$-distance $W_{1}$.

Given two probability measures $\mu$ and $\nu$ on $\mathcal{P}(V)$, recall that
$$
W_1(\mu,\nu)= \inf_{\pi \in \Pi(\nu_0,\nu_{1})} \iint d(x,y)\,\pi(dx \ dy)=\inf_{X\sim \mu, Y\sim \nu} \E [d(X,Y)]
$$

The following result asserts that $(\nu_t^\pi)_{t \in [0,1]}$ is actually a geodesic for $W_1$ when $\pi$ is an optimal coupling.

\begin{prop} \label{prop:georgia}
For any probability measures $\nu_0,\nu_1 \in \mathcal{P}(V)$, it holds
$$
W_1(\nu_s^{\pi^*},\nu_t^{\pi^*}) = |t-s| W_1(\nu_0,\nu_1) \qquad \forall s,t \in [0,1]
$$
where $\pi^*$ is an optimal coupling 
in the definition of $W_1(\nu_0,\nu_1)$ and where $\nu_t^{\pi^*}$ is defined in \eqref{atlantabis}.
\end{prop}

\begin{proof}
Fix two probability measures $\nu_0$, $\nu_1 \in \mathcal{P}(V)$ and $\pi^*$ an optimal coupling 
in the definition of $W_1(\nu_0,\nu_1)$ (since $\mathcal{P}(V)$ is compact $\pi^*$ is well defined).
For brevity, set $\nu_t:= \nu_t^{\pi^*}$. 

First, we claim that it is enough to prove that 
\begin{equation} \label{parigi}
W_1(\nu_s,\nu_t) \leq (t-s) W_1(\nu_0,\nu_1), \qquad \forall s,t \in [0,1] \mbox{ with } s \leq t .
\end{equation}
Indeed, assume \eqref{parigi}, then recalling that $W_1$ is a distance (see e.g. \cite{villani}), by the triangle inequality we have
\begin{align*}
W_1(\nu_0,\nu_1) 
& \leq 
W_1(\nu_0,\nu_s) + W_1(\nu_s,\nu_t) + W_1(\nu_t,\nu_1) 
 \leq 
s W_1(\nu_0,\nu_1) + (t-s) W_1(\nu_0,\nu_1) + t W_1(\nu_0,\nu_1) \\
& \leq 
W_1(\nu_0,\nu_1) .
\end{align*}
Hence, all the inequalities used above are actually equalities, which guarantees the conclusion of the proposition and hence the claim. 

Now, we prove \eqref{parigi}.
%{\it First proof:} \textcolor{green}{The advantage of this proof is to construct an optimal coupling for $W_1(\nu_s,\nu_t)$.} 
Let $(X,Y)$ be a random couple of law $\pi^*$. Fix $s \leq t$, it suffises to construct a random couple $(X_s, X_t)$ with marginal laws $\nu_s$ and $\nu_t$ so that 
$$\E[d(X_s,X_t)]\leq (t-s)\E[d(X,Y)]= (t-s) W_1(\nu_0,\nu_1).$$
From the last observation, let us remark that such a  couple $(X_s, X_t)$ will therefore realized $$\E[d(X_s,X_t)]=W_1(\nu_s,\nu_t).$$

Let $\Bigl((U_s^i,V_t^i)\Bigr)_{i\geq 1}$ be   an independent identically distributed sequence of random couples in $\{0,1\}^2$,  independent of $X$ and $Y$. We chose the law of $(U_s^1,V_t^1)$  given by
$$\P((U_s^1,V_t^1)=(0,0))=1-s,\quad \P((U_s^1,V_t^1)=(0,1))=0,$$
$$\quad \P((U_s^1,V_t^1)=(1,0))=t-s,\quad \P((U_s^1,V_t^1)=(1,1))=t,$$
so that $ U_s^1$ and $V_t^1$ are  Bernoulli random variables with respective parameters $s$ and $t$,
and we have 
$$\E(|U_s^1-V_t^1|)=(t-s) .$$
Given $(X,Y)=(x,y)$, with $x,y\in V$, let $(N_s ,N_t)$ denote the random couple defined by 
$$N_s= \sum_{i=1}^{d(x,y)} U_s^i, \quad N_t= \sum_{i=1}^{d(x,y)} V_s^i.$$
Then the laws of  $N_s$ and $N_t$ given $(X,Y)=(x,y)$ are respectively  $\mathcal{B}(d(x,y),s)$ and $\mathcal{B}(d(x,y),t)$, the binomial distribution with parameters $d(x,y)$, $s$ and $t$ respectively.

Finally,  given $(X,Y)=(x,y)$, with  $x,y\in V$, let $\Gamma$ denote a random geodesic chosen uniformly in $\Gamma(x,y)$, independently of the sequence $\left((U_s^i,V_t^i)\right)_{i\geq 1}$,
and  let $X_s= \Gamma_{N_s}$ be the random position on $\Gamma$ after $N_s$ jumps and  $X_t= \Gamma_{N_t}$ be the random position on $\Gamma$ after $N_t$ jumps. By definition, the law of $X_s$ and $X_t$ are respectively $\nu_s$ and $\nu_t$ and one has  $d(X_s,X_t)= |N_s-N_t|$.
Moreover, according to this construction, one has 
\begin{align*}
\E[ d(X_s,X_t)]&=\E\left[|N_s-N_t| \right]
= \E\left[\left|\sum_{i=1}^{d(X,Y)} U_s^i-\sum_{i=1}^{d(X,Y)} V_t^i\right|\right]\\
&\leq \E\left[\sum_{i=1}^{d(X,Y)} \left|U_s^i-V_t^i\right|\right]
%&=\E\left[\E\left[\left. \sum_{i=1}^{d(x,y)} \E\left[\left|U_s^i-V_t^i\right|\right] \right|X=x,Y=y \right]\right]\\
= \E\left[\sum_{i=1}^{d(X,Y)} \E\left[\left|U_s^i-V_t^i\right|\right]\right]
= (t-s) \E[d(X,Y)].
\end{align*}
This completes the proof of \eqref{parigi} and Proposition \ref{prop:georgia}.
\end{proof}

%%%%%%%%%%%%%%%%%%%%%%%%%%%%%%%%%%%%%%%%%%%%%%%%%%%%%%%%%%%%%%%%%%%%%%%%%%%%%%%%

\subsection{Differentiation property} \label{sec:differentiation}

A second property of the path defined in \eqref{atlanta} and \eqref{atlantabis} is the following time differentiation  property. 

For any $z$ on a given geodesic $\gamma$ from
 $x$ to $y$, if $z\neq y$, let $\gamma_+(z)$ denotes the (unique) vertex on $\gamma$ at distance $d(z,y)-1$ from $y$ (and thus at distance $d(x,z)+1$ from $x$), and similarly if $z\neq x$,
  let $\gamma_-(z)$ denote the vertex on $\gamma$ at distance $d(z,y)+1$ from $y$ (and hence at distance $d(x,z)-1$ from $x$).
  In other words, following the geodesic $\gamma$ from $x$ toward $y$, $\gamma_-(z)$ is the vertex just anterior to $z$, and $\gamma_+(z)$ the vertex posterior to $z$.
  
 For any real function $f$ on $V$,  we also define  two related notions of gradient along $\gamma$: for all $z\in \gamma$, $z\neq y$,
 $$
 \nabla_\gamma^+f(z)= f(\gamma_+(z))-f(z),
 $$ 
 and for all $z\in \gamma$, $z\neq x$,
 $$
 \nabla_\gamma^-f(z)= f(z)-f(\gamma_-(z)).
 $$
 By convention, we put $\nabla^-_\gamma f(x)=\nabla_\gamma^+f(y)=0$, and $\nabla_\gamma^+f(z)=\nabla^-_\gamma f(z)=0,$ if $z\notin \gamma.$
 Let $\nabla_\gamma f$ denote the following convex combination of these two gradients:
 $$
 \nabla_\gamma f(z)= \frac{d(y,z)}{d(x,y)} \nabla_\gamma^+f(z) + \frac{d(x,z)}{d(x,y)} \nabla_\gamma^-f(z) .
% \frac{d(y,z)}{d(x,y)} \nabla_\gamma^+f(z)+\frac{d(x,z)}{d(x,y)} \nabla_\gamma^-f(z).
 $$
 Observe that, although not explicitly stated, $\nabla_\gamma $ depends on $x$ and $y$.
 Finally, for all $z\in \llbracket x,y\rrbracket $, we define
 $$
 \nabla_{x,y}f (z) = \frac{1}{|\Gamma(x,z,y)| }\sum_{\gamma\in\Gamma(x,z,y)}  \nabla_\gamma f(z),
 $$
 and when $z\notin \llbracket x,y\rrbracket$, we set $\nabla_{x,y}f(z)=0.$
 
\begin{prop}  \label{paris}
For all  function $f \colon V \to \R $ and all $x,y \in V$, it holds
$$
\frac{\partial}{\partial t} \nu_t^{x,y}(f) =d(x,y) \nu_t^{x,y} ( \nabla_{x,y}f ) .%= \sum_{z\in G} \nabla_{x,y}f(z) \nu_t^{x,y}(z) .
$$
\end{prop}

As a direct consequence of the above differentiation property, we are able to give an explicit expression of the derivative (with respect to time) of the relative entropy of $\nu_t^\pi$ with respect to an arbitrary reference measure.

\begin{cor} \label{parisbis}
Let $\nu_0$, $\nu_1$ and $\mu$ be three probability measures on $V$. Assume that 
$\nu_0, \nu_1$ are absolutely continuous with respect to $\mu$. 
Then, for any coupling $\pi \in \Pi(\nu_0,\nu_1)$, it holds
$$
\frac{\partial}{\partial t} H(\nu_t^\pi | \mu)_{|_{t=0}} 
= 
\sum_{\genfrac{}{}{0pt}{}{x,z \in V:}{z \sim x}} \left( \log  \frac{\nu_0(z)}{\mu(z)} - \log \frac{\nu_0(x)}{\mu(x)}  \right)
\sum_{y \in V} d(x,y)\frac{|\Gamma(x,z,y)|}{|\Gamma(x,y)|} \pi(x,y).
$$
\end{cor}

The proof of Corollary \ref{parisbis} can be found below, while some example applications 
will be given in the next subsection. In order to prove Proposition \ref{paris}, we need some preparation.
Recall that $\mathcal{B}(n,t)$ denotes a binomial variable of parameter $n$ and $t$, and that, for any function 
$h \colon \{0,1,\ldots,n\} \to \R$, $\mathcal{B}(n,t)(h)=\sum_{k=0}^n h(k)C_n^kt^{k}(1-t)^{n-k}$.

\begin{lem}\label{deriv-binom}
Let $n\in \N^*$ and $t\in [0,1]$. For any function $h \colon \{0,1,\ldots,n\}\to\R$ it holds
$$
\frac{\partial}{\partial t} \mathcal{B}(n,t) (h) 
= 
\sum_{k=0}^n \left[(h(k+1)-h(k))(n-k)+(h(k)-h(k-1))k\right]\,C_n^kt^k(1-t)^{n-k},
$$
with the convention that $h(-1)=h(n+1)=0.$
\end{lem}
\begin{proof}[Proof of Lemma \ref{deriv-binom}]
By differentiating in $t$, we have 
\begin{align*}
\frac{\partial}{\partial t} \mathcal{B}(n,t)(h)
&= \sum_{k=0}^n h(k)kC_n^kt^{k-1}(1-t)^{n-k} - \sum_{k=0}^n h(k)(n-k)C_n^kt^{k}(1-t)^{n-k-1} .
\end{align*}
Now, using that $1=t+(1-t)$ and that $kC_n^k= (n-k+1) C_n^{k-1}$, we get
\begin{align*}
k C_{n}^{k} t^{k-1}(1-t)^{n-k} 
& = 
kC_{n}^{k} t^{k}(1-t)^{n-k} + 
(n-k+1)C_{n}^{k-1} t^{k-1}(1-t)^{n-k+1} ,
\end{align*}
with the convention that $C_{n}^{-1}=0$.
Similarly, using that $(n-k)C_n^{k}= (k+1) C_n^{k+1}$, we have
\begin{align*}
(n-k) C_{n}^{k} t^{k}(1-t)^{n-k-1} 
& = 
(n-k) C_{n}^{k} t^{k}(1-t)^{n-k}+(k+1)C_{n}^{k+1} t^{k+1}(1-t)^{n-k-1} .
\end{align*}
Hence, 
\begin{align*}
\frac{\partial}{\partial t}\mathcal{B}(n,t)(h)&=
\sum_{k=0}^n h(k)
(n-k+1)C_{n}^{k-1} t^{k-1}(1-t)^{n-k+1}-\sum_{k=0}^n h(k)(n-k)C_{n}^{k} t^{k}(1-t)^{n-k}\\
&\qquad+\sum_{k=0}^n h(k)kC_{n}^{k} t^{k}(1-t)^{n-k}-\sum_{k=0}^n h(k)(k+1)C_{n}^{k+1}t^{k+1}(1-t)^{n-k-1}\\
&= \sum_{k=0}^{n} (h(k+1)-h(k))(n-k)C_n^kt^k(1-t)^{n-k} + \sum_{k=0}^{n} (h(k)-h(k-1))kC_n^kt^k(1-t)^{n-k},
\end{align*}
with the convention that $h(-1)=h(n+1)=0$.
\end{proof}

We were informed by E. Hillion that the above elementary lemma also appears in his thesis \cite{hillionthesis}.
We are now in a position to prove Proposition \ref{paris}.

\begin{proof}[Proof of Proposition \ref{paris}]
Set $n=d(x,y)$ and let $\Gamma$ be a random variable uniformly distributed on $\Gamma(x,y)$ and $N_t$ be a random variable with Binomial law $\mathcal{B}(n,t)$ independent of $\Gamma$. By definition $\nu_t^{x,y}$ is the law of $X_t = \Gamma_{N_t}.$
Using the independence, we have
\begin{align*}
\nu_t^{x,y}(f)= \E\left[ f(X_t)\right] = \sum_{k=0}^n h(k) C^k_nt^k(1-t)^{n-k},
\end{align*}
with $h(k)=\E[f(\Gamma_k)]$, $k=0,1\dots,n$.
According to Lemma \ref{deriv-binom}, we thus get
\begin{align*}
\frac{\partial}{\partial t}\nu_t^{x,y}(f)&= \sum_{k=0}^n \left[(h(k+1)-h(k))(n-k)+(h(k)-h(k-1))k\right]\,C_n^kt^k(1-t)^{n-k}\\
&=\E\left[(h(N_t+1)-h(N_t))(n-N_t)+(h(N_t)-h(N_t-1))N_t \right]\\
&=\E\left[(f(\Gamma_{N_t+1})-f(\Gamma_{N_t}))d(\Gamma_{N_t},y)+(f(\Gamma_{N_t})-f(\Gamma_{N_t-1}))d(x,\Gamma_{N_t}) \right]\\
&=\E\left[(f(\Gamma^+(X_t))-f(X_t))d(X_t,y)+(f(X_t)-f(\Gamma^-(X_t)))d(x,X_t) \right] \\
&=\E\left[d(x,y)\nabla_\Gamma f(X_t) \right].
\end{align*}
Finally, observe that the law of $\Gamma$ knowing $X_t=z\in \llbracket x,y\rrbracket$ is uniform on $\Gamma(x,z,y).$ Indeed,
\begin{align*}
\P(\Gamma =\gamma,\ X_t=z)&= \P(\Gamma = \gamma,\ \gamma_{N_t}=z) = \P(\Gamma = \gamma,\ N_t=d(x,z),\ z\in \gamma)
 = \frac{\1_{\Gamma(x,z,y)}(\gamma)}{|\Gamma(x,y)|}\P(N_t=d(x,z)).
\end{align*}
On the other hand,
$$\P(X_t=z)=\nu_t^{x,y}(z)=\P(N_t=d(x,z)) \frac{|\Gamma(x,z,y)|}{|\Gamma(x,y)|},$$
which proves the claim. 
By the definition of $\nabla_{x,y}f$, it thus follows that
$$\frac{\partial}{\partial t}\nu_t^{x,y}(f) = d(x,y)\,\nu_t^{x,y} (\nabla_{x,y}f),$$
which completes the proof.
%Finally observe that the first sum in the right hand side of the latter can be rewritten as follows
%$$
%\sum \frac{\1_{z\in\gamma} f(\gamma_+(z))}{|\Gamma(x,y)|}d(y,z)C_{d(x,y)}^{d(x,z)} t^{d(x,z)}(1-t)^{d(y,z)} ,
%$$
%while the last one equals
%$$
%\sum \frac{\1_{z\in\gamma} f(\gamma_-(z))}{|\Gamma(x,y)|}d(x,z)C_{d(x,y)}^{d(y,z)} t^{d(x,z)}(1-t)^{d(y,z)} .
%$$
%The expected result follows.
%\begin{align*}
%&=\sum_{z\in\Omega}\sum_{\gamma\in\Gamma(x,y)} \frac{\1_{z\in\gamma} f(\gamma_+(z))}{|\Gamma(x,y)|}d(y,z)C_{d(x,y)}^{d(x,z)} t^{d(x,z)}(1-t)^{d(y,z)}\\
%&\quad+\sum_{z\in\Omega}\sum_{\gamma\in\Gamma(x,y)} \frac{\1_{z\in\gamma} f(z)}{|\Gamma(x,y)|}d(x,z)C_{d(x,y)}^{d(x,z)} t^{d(x,z)}(1-t)^{d(y,z)}\\
%&\quad-\sum_{z\in\Omega}\sum_{\gamma\in\Gamma(x,y)} \frac{\1_{z\in\gamma} f(z)}{|\Gamma(x,y)|}d(y,z)C_{d(x,y)}^{d(x,z)} t^{d(x,z)}(1-t)^{d(y,z)}\\
%&\quad-\sum_{z\in\Omega}\sum_{\gamma\in\Gamma(x,y)} \frac{\1_{z\in\gamma} f(\gamma_-(z))}{|\Gamma(x,y)|}d(x,z)C_{d(x,y)}^{d(y,z)} t^{d(x,z)}(1-t)^{d(y,z)}.\\
%\end{align*}
\end{proof}

\begin{proof}[Proof of Corollary \ref{parisbis}]
For simplicity, let $F=\log (\nu_0/\mu)$. Observe that, since $\nu_0$ and $\nu_1$ are absolutely continuous with respect to $\mu$, so is $\nu_t^\pi$.
Now we observe that, since $\sum_{z \in V} \frac{\partial}{\partial t} \nu_t^\pi(z)=0$,  by Proposition \ref{paris}
(recall that $\nu_0^\pi=\nu_0$ and $\nu_0^{x,y}=\delta_x$ by construction),
\begin{align*}
\frac{\partial}{\partial t} H(\nu_t^\pi | \mu)_{|_{t=0}} 
& = 
\frac{\partial}{\partial t} \left( \sum_{z \in V}  \nu_t^\pi(z)  \log \frac{\nu_t^\pi(z)}{\mu(z)} \right)_{|_{t=0}}
=
\frac{\partial}{\partial t} \nu_t^\pi (F)_{|_{t=0}} 
 =
\sum_{(x,y) \in V^2} \pi(x,y) \frac{\partial}{\partial t} \nu_t^{x,y}(  F ) \\
& =
\sum_{(x,y) \in V^2} \pi(x,y) d(x,y) \nabla_{x,y}   F  (x) .
\end{align*}
By the definition of the gradient, for any $\gamma \in \Gamma(x,y)$, it holds
$\nabla_\gamma   F(x)
=
 \nabla_\gamma^+   F(x)$. Thus, by the definition of $\nabla_{x,y}F$, we get
\begin{align*}
\frac{\partial}{\partial t} H(\nu_t^\pi | \mu)_{|_{t=0}} 
& =
\sum_{(x,y) \in V^2} \frac{\pi(x,y) d(x,y)}{|\Gamma(x,y)|} \sum_{\gamma \in \Gamma(x,y)}  \nabla_\gamma^+ F (x).
%& =
%\sum_{x \in V} \sum_{z \sim x} (\log f(z) - \log f(x) ) \sum_{y \in V} \frac{\pi(x,y) d(x,y)}{|\Gamma(x,y)|} 
%\sum_{\gamma \in \Gamma}   \1_{\gamma \in \Gamma(x,y)} \1_{\gamma^+(x)=z} .
\end{align*} 
%By \eqref{sophiacardinale}, we have for any  $x,y \in V$ and any $z \sim x$
%$$
% \sum_{\gamma \in \Gamma}   \1_{\gamma \in \Gamma(x,y)} \1_{\gamma^+(x)=z}
% =
% \card \{\gamma \in \Gamma(x,y) : z \in \gamma \} \1_{z\in \llbracket x,y\rrbracket}
%=
%|\Gamma(z,y)| \1_{z\in \llbracket x,y\rrbracket}.
%$$
Now, observe that for $(x,y)\in V^2$ given, it holds
\begin{align*}
\sum_{\gamma \in \Gamma(x,y)}\nabla_\gamma^+F(x)&=\sum_{\gamma \in \Gamma(x,y)}F(\gamma^+(x))-F(x)=\sum_{z\sim x} (F(z)-F(x)) |\Gamma(x,z,y)|\,,
\end{align*}
completing the proof.
\end{proof}

%%%%%%%%%%%%%%%%%%%%%%%%%%%%%%%%%%%%%%%%%%%%%%%%%%%%%%%%%%%%%%%%%%%%%%%%%%%%%%%%%%%%%%%%

\subsection{Tensoring property}

In this section we prove that the path $(\nu_t^{x,y})_{t\in[0,1]}$ constructed in Section \ref{construction} does tensorise. This will appear to be crucial in deriving the displacement convexity of the entropy on product spaces. Moreover we shall prove that, in order to have this tensoring property, the law of the random variable $N_t$ introduced in the construction of the path $(\nu_t^{x,y})_{t\in[0,1]}$,  must be, modulo a change of time, a binomial (see Proposition \ref{loident} below).  
The tensoring property of  the path $(\nu_t^{x,y})_{t\in[0,1]}$  is the following.

\begin{lem} \label{tensor}
Let $G_1=(V_1,E_1)$, $G_2=(V_2,E_2)$ be two graphs and let $G=G_1 \boxempty G_2$ be their Cartesian product.
Then, for any $x=(x_1,x_2)$, $y=(y_1,y_2)$ and $z=(z_1,z_2)$ in $V_1 \times V_2$,
$$
\nu^{x,y}_t(z)=\nu^{x_1,y_1}_t(z_1)\nu^{x_2,y_2}_t(z_2).
$$
\end{lem}

\begin{proof}
Fix $x=(x_1,x_2)$, $y=(y_1,y_2)$ and $z=(z_1,z_2)$ in $V_1 \times V_2$. Then, we observe that, given two geodesics, one from $x_1$ to $y_1$, and one from $x_2$ to $y_2$, one can construct exactly $C_{d(x,y)}^{d(x_1,y_1)}$ different geodesics from $x$ to $y$ (by choosing the $d(x_1,y_1)$ positions where to change the first coordinate, according to the geodesic joining $x_1$ to $y_1$, and thus changing the second coordinate in the remaining $d(x_2,y_2)=d(x,y)-d(x_1,y_1)$ positions, according to the geodesic joining $x_2$ to $y_2$). This construction exhausts all the geodesics from $x$ to $y$. 
Hence,
\begin{equation}\label{eq:geod-prod}
|\Gamma(x,y)| = C_{d(x,y)}^{d(x_1,y_1)} |\Gamma(x_1,y_1)| \times  |\Gamma(x_2,y_2)|.
\end{equation}
Observe also that $z$ belongs to some geodesic from $x$ to $y$ if and only if $z_1$ and $z_2$ belong respectively to some geodesic from $x_1$ to $y_1$, and from $x_2$ to $y_2$. Therefore, by \eqref{sophiacardinale}, it follows that
$$|\Gamma(x,z,y)| = C_{d(x,z)}^{d(x_1,z_1)}C_{d(z,y)}^{d(z_1,y_1)}|\Gamma(x_1,z_1,y_1)| \times |\Gamma(x_2,z_2,y_2)|.$$
So, it holds that
\begin{align*}
\nu_t^{x,y}(z) & = C_{d(x,y)}^{d(x,z)} t^{d(x,z)} (1-t)^{d(y,z)}\; \frac{|\Gamma(x,z,y)| } {|\Gamma(x,y)|}    \\
& = 
\frac{C_{d(x,y)}^{d(x,z)} C_{d(x,z)}^{d(x_1,z_1)} C_{d(y,z)}^{d(y_1,z_1)} } {C_{d(x,y)}^{d(x_1,y_1)}} 
t^{d(x_1,z_1)} (1-t)^{d(y_1,z_1)} \frac{|\Gamma(x_1,z_1,y_1)| } {|\Gamma(x_1,y_1)|}
t^{d(x_2,z_2)} (1-t)^{d(y_2,z_2)} \frac{|\Gamma(x_2,z_2,y_2)|} {|\Gamma(x_2,y_2)|} 
\\
& =
\nu^{x_1,y_1}_t(z_1)\nu^{x_2,y_2}_t(z_2)\,,
\end{align*}
where we used that $d(x,z)=d(x_1,z_1)+d(x_2,z_2)$, and similarly for $d(y,z)$, and the fact (that the reader can easily verify) that
$$
\frac{C_{d(x,y)}^{d(x,z)} C_{d(x,z)}^{d(x_1,z_1)} C_{d(y,z)}^{d(y_1,z_1)} } {C_{d(x,y)}^{d(x_1,y_1)}} =
C_{d(x_1,y_1)}^{d(x_1,z_1)} C_{d(x_2,y_2)}^{d(x_2,z_2)} .
$$
%This ends the proof.
\end{proof}

\begin{prop}\label{loident}
In the construction of $\nu_t^{x,y}$, $t \in [0,1]$, use a general random variable $N_t^{d(x,y)} \in \{0,1,\dots,d(x,y)\}$, of parameter $d(x,y)$ and $t$, that satisfies a.s. $N_0^{d(x,y)} = 0$ and $N_1^{d(x,y)} = d(x,y)$ (instead of the Binomial, observe that this condition is here to ensure that $\nu_0^{x,y}=\delta_x$ and $\nu_1^{x,y}=\delta_y$, namely that $\nu_t^{x,y}$ is still an interpolation between the two Dirac measures) , so that
\begin{equation*} %\label{atlanta}
\nu_t^{x,y}(z)= \P\left( N_t^{d(x,y)} = d(x,z) \right) \frac{| \Gamma(x,z,y)|}{| \Gamma(x,y)|} .
\end{equation*}
Let $G_1=(V_1,E_1)$, $G_2=(V_2,E_2)$ be two graphs and let $G=G_1 \boxempty G_2$ be their Cartesian product.
Assume that for any $x=(x_1,x_2)$, $y=(y_1,y_2)$ and $z=(z_1,z_2)$ in $V_1 \times V_2$,
$$
\nu^{x,y}_t(z)=\nu^{x_1,y_1}_t(z_1)\nu^{x_2,y_2}_t(z_2) \qquad \forall t \in [0,1].
$$
Then, there exists a function $a \colon [0,1] \to [0,1]$ with $a(0)=0$, $a(1)=1$, such that $N_t^{d(x,y)} \sim \mathcal{B}(a(t),d(x,y))$.
\end{prop}

\begin{proof}
Following the proof of Lemma \ref{tensor} we have,
\begin{align*}
\nu_t^{x,y}(z) & = \P\left( N_t^{d(x,y)} = d(x,z) \right) \frac{|\Gamma(x,z,y) |} {| \Gamma(x,y)|}   \\
& = 
\frac{C_{d(x,z)}^{d(x_1,z_1)} C_{d(y,z)}^{d(y_1,z_1)} } {C_{d(x,y)}^{d(x_1,y_1)}} \P\left( N_t^{d(x,y)} = d(x,z) \right) 
\;
 \frac{| \Gamma(x_1,z_1,y_1) |} {| \Gamma(x_1,y_1)|} 
\;
 \frac{| \Gamma(x_2,z_2,y_2) |} {| \Gamma(x_2,y_2)|} \,.
\end{align*}
On the other hand,
$$
\nu^{x_1,y_1}_t(z_1) =  \P\left( N_t^{d(x_1,y_1)} = d(x_1,z_1) \right) \frac{| \Gamma(x_1,z_1,y_1) |} {| \Gamma(x_1,y_1)|} 
$$
and
$$
\nu^{x_2,y_2}_t(z_2) =  \P\left( N_t^{d(x_2,y_2)} = d(x_2,z_2) \right) \frac{| \Gamma(x_2,z_2,y_2)| } {| \Gamma(x_2,y_2)|} 
.
$$
Hence, the identity $\nu^{x,y}_t(z)=\nu^{x_1,y_1}_t(z_1)\nu^{x_2,y_2}_t(z_2)$ ensures that
$$
\frac{C_{d(x,z)}^{d(x_1,z_1)} C_{d(y,z)}^{d(y_1,z_1)} } {C_{d(x,y)}^{d(x_1,y_1)}} \P\left( N_t^{d(x,y)} = d(x,z) \right) 
=
\P\left( N_t^{d(x_1,y_1)} = d(x_1,z_1) \right) \P\left( N_t^{d(x_2,y_2)} = d(x_2,z_2) \right)
$$
for any $z_1\in \llbracket x_1,y_1\rrbracket$, $z_2\in \llbracket x_2,y_2\rrbracket$.

Now, observe that
$$
\frac{C_{d(x,z)}^{d(x_1,z_1)} C_{d(y,z)}^{d(y_1,z_1)} } {C_{d(x,y)}^{d(x_1,y_1)}}
=
\frac{ C_{d(x_1,y_1)}^{d(x_1,z_1)} C_{d(x_2,y_2)}^{d(x_2,z_2)}}{C_{d(x,y)}^{d(x,z)}} .
$$
Hence, the latter can be rewritten as
$$
\frac{\P\left( N_t^{d(x,y)} = d(x,z) \right)}{C_{d(x,y)}^{d(x,z)}}
=
\frac{\P\left( N_t^{d(x_1,y_1)} = d(x_1,z_1) \right)}{C_{d(x_1,y_1)}^{d(x_1,z_1)}} 
\times 
\frac{ \P\left( N_t^{d(x_2,y_2)} = d(x_2,z_2) \right)}{C_{d(x_2,y_2)}^{d(x_2,z_2)}} .
$$
Set, for simplicity, for any $n,k$, $0 \leq k \leq n$
$$
p_{n,k} := \frac{\P\left( N_t^{n} = k \right)}{C_n^k} .
$$
Notice that $p_{n,k}$ depends also on $t$, while not explicitly stated.
We end up with the following induction formula
\begin{equation} \label{induction}
p_{n,k} = p_{n_1,k_1} \cdot  p_{n-n_1,k-k_1}
\end{equation}
for any integers $k_1,n_1, k,n$ satisfying the following conditions
$$
 k,n_1 \leq n, \qquad k_1 \leq \min(k,n_1), \qquad \mbox{and} \quad n_1-k_1 \leq n-k .
$$
(We set, $n=d(x,y)$, $n_1=d(x_1,y_1)$, $k=d(x,z)$ and $k_1=d(x_1,z_1)$).

The special choice $n_1=1$, $k_1=0$ leads to
\begin{equation}\label{1}
p_{n,k} = p_{1,0} \cdot  p_{n-1,k} .
\end{equation}
Hence, it cannot be that $p_{1,0}=0$ (otherwise we would have $p_{n,k}=0$ for any $k \geq 0$, any $n \geq 1$, which clearly is impossible
since $\sum_{k=0}^n C_n^k p_{n,k} = 1$).

%On the other hand, the special choice $n=n_1=1$, $k=k_1=0$ leads to $p_{1,0}=p_{1,0}p_{0,0}$. In turn, since 
%$p_{1,0} \neq 0$, $p_{0,0}=1$.

Set $b=b(t)=p_{1,0}$. From \eqref{1} we deduce that
$$
p_{n,k} = b^{n-k} p_{k,k} .
$$
Finally, the special choice $n=k$, $n_1=k_1=k-1$, in \eqref{induction}, ensures that
$$
p_{k,k} = p_{k-1,k-1} \cdot p_{1,1} .
$$
Since $p_{1,0}+p_{1,1} = 1$, the latter reads as
$$
p_{k,k} = p_{1,1}^k = (1-b)^k .
$$
It follows that
$$
p_{n,k} = b^{n-k} (1-b)^k  \qquad \forall n, \; \forall  k \leq n .
$$

Now set $a(t)=1-b(t)$ to end up with
$$
\P\left( N_t^{n} = k \right) = C_n^k a^k (1-a)^{n-k} \,,
$$
which guarantees that $N_t^{d(x,y)}$ is indeed a binomial variable of parameter $a(t)$ and $d(x,y)$.

To end the proof, it is suffices  to observe that $N_0^{d(x,y)}=0$ implies  $a(0)=0$, and that $N_1^{d(x,y)}=d(x,y)$ implies $a(1)=1$.
\end{proof}

%%%%%%%%%%%%%%%%%%%%%%%%%%%%%%%%%%%%%%%%%%%%%%%%%%%%%%%%%%%%%%%%%%%%%%%%%%%%%%%%%%%%%%%%%%%%%%

\subsection{Examples}

In this section we collect some elementary facts on specific examples. Namely we give explicit expressions of $\nu_t^{x,y}$, and derive some properties, when available, on  the complete graph, the two-point space, and  the hypercube.
% and on the slices of the cube.

\subsubsection{Complete graph $K_n$} \label{sec:complete}

Let $K_n$ be the complete graph with $n$ vertices. Then, given any two points $x,y \in K_n$, there  exists only one geodesic from $x$ to $y$, namely $\Gamma(x,y)=\{(x,y)\}$.
Hence, by construction of $\nu_t^{x,y}$, we have 
\begin{equation} \label{marne}
\nu_t^{x,y}(z)=0 \;\forall z \neq x,y; \quad \nu_t^{x,y}(x)= 1-t, \quad \mbox{and} \quad  \nu_t^{x,y}(y)= t.
\end{equation}
Therefore, for any coupling $\pi$ with marginals $\nu_0$ and $\nu_1$ (two given probability measures on $K_n$), we have
for any $z \in K_n$,
\begin{align*}
\nu_t^\pi(z) 
& = 
\sum_{(x,y) \in C(z)} \nu_t^{x,y}(z) \pi(x,y) 
=
\sum_{y \in K_n} \nu_t^{z,y}(z) \pi(z,y) + \sum_{x \in K_n} \nu_t^{x,z}(z) \pi(x,z) \\
& =
(1-t) \sum_{y \in K_n} \pi(z,y) + t \sum_{x \in K_n} \pi(x,z) 
 =
(1-t) \nu_0(z) + t \nu_1(z) .
\end{align*}
As a conclusion, on the complete graph, $\nu_t^\pi$ is a simple linear combination of $\nu_0$ and $\nu_1$ that does not depend on $\pi$.

Moreover, under the assumption of Corollary \ref{parisbis}, since $d(x,y)= |\Gamma(x,y)| = |\Gamma(z,y)|=1$,  we have
\begin{align*}
\frac{\partial}{\partial t} H(\nu_t^\pi | \mu)_{|_{t=0}}  
& = 
\sum_{x \in K_n} \sum_{z \sim x} (\log f(z) - \log f(x) )  \pi(x,z) 
 =
\sum_{z \in K_n} \log f(z)  \nu_1(z) - \sum_{x \in K_n} f(x) \log f(x)  \mu(x)
\end{align*}
where we set for simplicity $f=\nu_0/\mu$.  On the other hand, since $f$ is a density with respect to $\mu$,
\begin{align*}
- \mathcal{E}_\mu(f, \log f) 
& :=
- \frac{1}{2} \sum_{x,z \in K_n} (\log f(z) - \log f(x) )(f(z) -f(x)) \mu(x) \mu(z) \\
& = 
\sum_{z \in K_n} \log f(z)  \mu(z) - \sum_{x \in K_n} f(x) \log f(x)  \mu(x) .
\end{align*}
Hence, if $\nu_1 = \mu \equiv 1/n$ is the uniform measure on $K_n$ (notice all the measures on $K_n$ are then absolutely continuous with respect to $\mu$), we can conclude that
\begin{equation}\label{crepe}
\frac{\partial}{\partial t} H(\nu_t^\pi | \mu)_{|_{t=0}}  = - \mathcal{E}_\mu(f, \log f) .
\end{equation}
Note that, when $\mu \equiv 1/n$, $\mathcal{E}_\mu$ corresponds to the Dirichlet form associated to the uniform chain on the complete graph (each point can jumps to each point with probability $1/n$).

As a summary, on the complete graph we have:
\noindent For any coupling $\pi$, for any $t \in [0,1]$,
$$
\nu_t^\pi = (1-t)\nu_0 + t \nu_1 .
$$
For $\nu_1 = \mu \equiv 1/n$ and $f=\nu_0/\mu$, it holds
$$
\frac{\partial}{\partial t} H(\nu_t^\pi | \mu)_{|_{t=0}}  = - \mathcal{E}_\mu(f, \log f) .
$$

\subsubsection{The two-point space}\label{sec:two-point}

The previous computations apply in particular to the two-point space $\{0,1\}$.
In this specific case, let us consider $\mu$ to be a Bernoulli$(p)$ measure (\textit{i.e.}\ $\mu(1)=p=1-q=1-\mu(0)$). As above, $\nu_t^\pi = (1-t)\nu_0 + t \nu_1$, for any coupling $\pi$ of $\nu_0$ and $\nu_1$.  Moreover, it can also be checked by an easy computation that, for any $t \in [0,1]$,
$$  
\frac{\partial^2}{\partial t^2} H(\nu_t^\pi | \mu)  = \frac{C^2} {(\nu_0(0)+tC)(\nu_0(1)-tC)} \ge 4C^2\,,
$$
where $C =\nu_1(0)-\nu_0(0)$, and $\|\nu_0-\nu_1\|_{TV} = |\nu_1(0)-\nu_0(0)|$.
As a result, one arrives at the following displacement convexity of the entropy of $\nu^\pi_t$ on the two-point space:
\begin{equation}\label{eq:DC-2pt}
H(\nu^\pi_t|\mu) \leq (1-t)H(\nu_0|\mu) + t H(\nu_1|\mu) 
- 2 t(1-t) \|\nu_0 - \nu_1\|_{TV}^2, 
\qquad t \in [0,1] \,.
\end{equation}
In Section~\ref{sec:Section DC} below, we refine the above inequality further, and generalize in two ways -- by deriving displacement convexity of entropy on the complete graph and the $n$-dimensional hypercube. 

As an application, let us set $\nu_1=\mu$,  and use $f=\nu_0/\mu$ for the density; taking the limit $t \to 0$, and using
$$
\frac{\partial}{\partial t} H(\nu_t^\pi | \mu)_{|_{t=0}}  
= 
- \frac{pq}{2} (f(1)-f(0))(\log f(1) - \log f(0))
=:
 -  \mathcal{E}_\mu(f, \log f)\,,
$$ 
we get a {\em reinforced} modified logarithmic Sobolev inequality on the two-point space of the following type:
\begin{equation}\label{eq:RMLSI-2pt}
\ent_\mu(f) \le \mathcal{E}_\mu(f, \log f) - 2 \|f\mu - \mu\|_{TV}^2\,.
\end{equation}
In the above,  $\mathcal{E}_\mu(f, \log f)$ corresponds to the Dirichlet form associated with the Markov chain jumping from $0$ to $1$ with probability $p$ and from $1$ to $0$ with probability $q$.  The inequality is a reinforcement of a modified log-Sobolev inequality, considered by previous researchers (as mentioned in the introduction), which lacks the negative term. Similarly to \eqref{eq:DC-2pt}, we also refine \eqref{eq:RMLSI-2pt} further in Proposition~\ref{hwitwopoints}.

\subsubsection{The $n$-dimensional hypercube $\Omega_n$} \label{sec:hypercube}

Consider the $n$-dimensional hypercube $\Omega_n=\{0,1\}^n$ whose edges consist of pairs of vertices p that differ  in precisely one coordinate. The graph distance here coincides with the Hamming distance:
$$
d(x,y)= \sum_{i=1}^n \1_{x_i \neq y_i} ,\quad x,y \in \Omega_n.
$$
Then, one observes that $|\Gamma(x,y)|=d(x,y)!$ (since, in order to move from $x$ to $y$ in the shortest way, one just needs to choose, among $d(x,y)$ coordinates where $x$ and $y$ differ, the order of the flips ({\it i.e.}\ moves from $x_i$ to $1-x_i$)). It follows from \eqref{atlanta} that, as soon as $z$ belongs to a geodesic from $x$ to $y$,
$$
\nu_t^{x,y}(z) = C_{d(x,y)}^{d(x,z)} t^{d(x,z)}(1-t)^{d(y,z)} \frac{d(x,z)! d(y,z)!}{d(x,y)!} =
t^{d(x,z)}(1-t)^{d(y,z)},
$$
and $\nu_t^{x,y}(z)=0$ if $z$ does not belong to a geodesic from $x$ to $y$.

This expression can be recovered using the tensorisation property above. Namely, observe that Equation \eqref{marne} can be rewritten for the two-point space as follows, for all coordinates:
$$
\nu_t^{x_i,y_i}(z_i)= \1_{\{x_i,y_i\}}(z_i) t^{d(x_i,z_i)}(1-t)^{d(y_i,z_i)} .
$$
Hence, by Lemma \ref{tensor}, 
$$
\nu_t^{x,y}(z) = \prod_{i=1}^n \nu_t^{x_i,y_i}(z_i) = t^{d(x,z)}(1-t)^{d(y,z)}\,,
$$
as soon as $z$ belongs to a geodesic from $x$ to $y$, and $0$ otherwise. Observe that the latter can also be rewritten 
in terms of a product of probability measures on the fibers as
\begin{equation} \label{product}
\nu_t^{x,y} = \otimes_{i=1}^n ((1-t) \delta_{x_i} + t\delta_{y_i}) .
\end{equation}

Given two probability measures on $\Omega_n$, and a coupling $\pi$ on $\Omega_n \times \Omega_n$, we can finally define
$$
\nu_t^\pi (z) = \sum_{(x,y)\in \Omega_n^2} t^{d(x,z)}(1-t)^{d(y,z)} \pi(x,y) .
$$
On the $n$-dimensional hypercube we have:
\noindent for any couple $(x,y)\in \Omega_n^2$ and for any $t \in [0,1]$,
$$
\nu_t^{x,y} =  \sum_{z \in \llbracket x,y\rrbracket}t^{d(x,z)}(1-t)^{d(y,z)}\delta_z = \otimes_{i=1}^n ((1-t) \delta_{x_i} + t\delta_{y_i}).
$$

\section{Weak transport cost} \label{sec:wtc}

In this section we recall a notion of a discrete Wasserstein-type distance, called weak transport cost -- introduced and studied in \cite{marton96, samson},
%by the third author in \cite{samson} 
developed further in \cite{GRST} --  and collect some useful facts from \cite{GRST}.
Also, we introduce the notion of a  Knothe-Rosenblatt coupling which will play a crucial role in the displacement convexity of the entropy property on product spaces.

%%%%%%%%%%%%%%%%%%%%%%%%%%%%%%%%%%%%%%%%%%%%%%%%%%%%%%%%%%%%%%%%%%%%%%%%%%%%%%%%%%%%

\subsection{Definition and first properties}

For the notion of a weak transport cost, first recall the definition of $P(\nu_0,\nu_1)$ introduced in Section \ref{sec:notation}.

\begin{defi} \label{weak}
Let $\nu_0, \nu_1 \in \mathcal{P}(V)$. Then, the weak transport cost $\Tw_{2}(\nu_1 |\nu_0 )$ between $\nu_0$ and $\nu_1$ is defined as
$$
\Tw_{2}(\nu_1 |\nu_0 ):=\inf_{p \in P(\nu_0,\nu_1)} \sum_{x \in V} \left( \sum_{y \in V} d(x,y) p(x,y) \right)^2 \nu_0(x) .
$$

\end{defi}

It can be shown that $$(\nu_{0},\nu_{1})\mapsto \sqrt{\Tw_2(\nu_{1}|\nu_{0})} +\sqrt{\Tw_2(\nu_{0}|\nu_{1})}$$ is a distance on $\mathcal{P}(V)$, see \cite{GRST}.

Also recall from the introduction, the following notation: given $\pi \in \Pi(\nu_0,\nu_1)$, consider the kernels $p\in P(\nu_{0},\nu_{1})$ and $\bar{p}\in P(\nu_{1},\nu_{0})$ defined by $\pi(x,y)=\nu_0(x)p(x,y) = \nu_{1}(y)\bar{p}(y,x) $ and set
\begin{eqnarray}%\label{weakbis}
I_2(\pi):=  \sum_{x \in V} \left( \sum_{y \in V} d(x,y) p(x,y) \right)^2 \nu_0(x),
\end{eqnarray}
\begin{eqnarray*}
\bar{I}_2(\pi):=  \sum_{y \in V} \left( \sum_{x \in V} d(x,y) \bar{p}(y,x) \right)^2 \nu_1(y),
\end{eqnarray*}
and
\begin{eqnarray*}
J_2(\pi):=   \left(\sum_{x \in V} \sum_{y \in V} d(x,y)\pi(x,y) \right)^2.
\end{eqnarray*}
With this notation, 
$$\Tw_{2}(\nu_{0}|\nu_{1})=\inf_{\pi \in \Pi(\nu_{0},\nu_{1})}I_{2}(\pi) .
$$
Also, define
$$
\hat {\mathcal{T}}_{2}(\nu_0 ,\nu_1 ) :=\inf_{\pi \in \Pi(\nu_{0},\nu_{1})}J_{2}(\pi),
$$
and observe that $\hat {\mathcal{T}}_{2}(\nu_0 ,\nu_1 )=W_1^2(\nu_{0},\nu_{1})$ where $W_1$ is the usual $L_1$-Wasserstein distance associated to the distance $d$.

When $\nu_0$ and $\nu_1$ are absolutely continuous with respect to some probability measure $\mu$, and $d$ is the Hamming distance $d(x,y)=\1_{x \neq y}$, $x,y \in V$,  the weak transport cost  and the   $L_1$-Wasserstein distance
take an explicit form. This is stated in the next lemma. We give the proof for completeness.

\begin{lem}[\cite{GRST}] \label{lem:postive}
Assume that $\nu_0, \nu_1 \in \mathcal{P}(V)$ are absolutely continuous with respect to a third probability measure 
$\mu \in  \mathcal{P}(V)$, with respective densities $f_0$ and $f_1$. Assume that $d(x,y)=\1_{x \neq y}$, $x,y \in V$. Then it holds
$$
\Tw_2(\nu_1 |\nu_0 ) = \int \left[1-\frac{f_1}{f_0} \right]_+^2 f_0\,d\mu 
$$
where $[X]_+=\max(X,0)$, and 
$$
\sqrt{\hat {\mathcal{T}}_{2}(\nu_0 ,\nu_1 )}=\int \left[f_0-f_1 \right]_+ \,d\mu =\frac 12 \int \left|f_0-f_1 \right| \,d\mu = \frac 12 \|\nu_0-\nu_1\|_{TV}
$$
with $\|\cdot  \|_{TV}$, the total variation norm.
\end{lem}

\begin{rem}
Observe that $\Tw_2(\nu_1 |\nu_0 )$ does not depend on $\mu$.
\end{rem}

\begin{proof} For any $\pi \in \Pi(\nu_{0},\nu_{1})$ and any $x\in V$, one has 
$$1- \sum_{y \in V} d(x,y)p(x,y)=\frac{\pi(x,x)}{\nu_0(x)}\leq \frac{\min(\nu_0(x),\nu_1(x))}{\nu_0(x)}=\min\left(\frac{f_1(x)}{f_0(x)},1\right).$$
and therefore 
$$\left[1-\frac{f_1(x)}{f_0(x)} \right]_+\leq \sum_{y \in V} d(x,y)p(x,y).$$
By integrating  with respect to the measure $\nu_0$ and then optimizing over all $\pi \in \Pi(\nu_{0},\nu_{1})$, it follows that 
$$\int \left[f_0-f_1 \right]_+ \,d\mu\leq\sqrt{\hat {\mathcal{T}}_{2}(\nu_0 ,\nu_1 )},$$
and 
$$
 \int \left[1-\frac{f_1}{f_0} \right]_+^2 f_0\,d\mu\leq 
\Tw_2(\nu_1 |\nu_0 ).$$
The equality is reached choosing $\pi^* \in \Pi(\nu_{0},\nu_{1})$ defined by
\begin{align}\label{couplagebis}
\pi^*(x,y) 
& =
\nu_0(x)p^*(x,y) %\nonumber \\
%& 
= 
\1_{x= y}\min(\nu_0(x),\nu_1(x))+\1_{x\neq y}\frac{[\nu_0(x)-\nu_1(x)]_+[\nu_1(y)-\nu_0(y)]_+}{ \sum_{z \in V} [\nu_1(z)-\nu_0(z)]_+},
\end{align}
since 
$\sum_{y \in V} d(x,y)p^*(x,y)=\left[1-\frac{f_1(x)}{f_0(x)} \right]_+.$
\end{proof}

%%%%%%%%%%%%%%%%%%%%%%%%%%%%%%%%%%%%%%%%%%%%%%%%%%%%%%%%%%%%%%%%%%%%%%%%%%%%%%%%%%%%

\subsection{The Knothe-Rosenblatt coupling}\label{sec:knothe}

In this subsection, we recall a general method, due to Kno\-the-Rosenblatt \cite{knothe,rosenblatt}, enabling to construct couplings between probability measures on product spaces.

Consider two graphs $G_1=(V_1,E_1)$ and $G_2=(V_2,E_2)$ and two probability measures 
$\nu_0,\nu_1 \in  \mathcal{P}(V_1 \times V_2)$. The disintegration formulas of $\nu_0, \nu_1$  (recall \eqref{disintegration}) read
\begin{equation} \label{disintegration2}
\nu_0(x_1,x_2)=\nu_0^2(x_2)\nu_0^1(x_1|x_2) \qquad \mbox{and}
\qquad \nu_1(y_1,y_2)=\nu_1^2(y_2)\nu_1^1(y_1|y_2) .
\end{equation}
Let $\pi^2 \in  \mathcal{P}( V_2^2)$ be a coupling 
of  $\nu_0^2$, $\nu_1^2$, and  for all $(x_{2},y_{2}) \in V_{2}^2$ let $\pi^1(\,\cdot\,|x_2,y_2) \in  \mathcal{P}(V_1^2)$ be a coupling of 
$\nu_0^1(\,\cdot\, | x_2)$ and $\nu_1^1(\,\cdot\, | y_2)$, $x_2, y_2 \in V_2$. We are now in a position to define the Knothe-Rosenblatt coupling.

\begin{defi}[Knothe-Rosenblatt coupling]
Let $\nu_0,\nu_1 \in  \mathcal{P}(V_1 \times V_2)$, and consider a family of couplings $\pi^2, \{\pi^1(\,\cdot\,|x_{2},y_{2})\}_{x_{2},y_{2}}$ as above; the coupling $\hat \pi \in \mathcal{P}([V_1 \times V_2]^2)$, defined by
$$
\hat \pi((x_1,x_2),(y_1,y_2)) := \pi^2(x_2,y_2) \pi^1(x_1,y_1|x_2,y_2)\,, \qquad (x_1,x_2),(y_1,y_2) \in V_1 \times V_2 
$$
is called the \emph{Knothe-Rosenblatt coupling} of $\nu_0, \nu_1$ associated with the family of couplings 
$$\left\{\pi^2, \{\pi^1(\, \cdot\, |x_{2}, y_{2}) \}_{x_{2}, y_{2}} \right\}.$$
\end{defi}
It is easy to check that the Knothe-Rosenblatt coupling is indeed a coupling of $\nu_0, \nu_1$.
Note that it is usually required that the couplings $\pi^2,\{\pi^1(\,\cdot\,|x_{2},y_{2})\}_{x_{2},y_{2}}$ are optimal for some weak transport cost, but we will not make this assumption in what follows.

The preceding construction can easily be generalized to products of $n$ graphs. 
Consider $n$ graphs $G_1=(V_1,E_1), \dots, G_n=(V_n,E_n)$, and two probability measures $\nu_0, \nu_1 \in \mathcal{P}(V_1 \times \cdots\times V_n)$ admitting the following disintegration formulas: for all $x=(x_1,\dots,x_n), y=(y_1,\dots,y_n) \in V_1 \times \dots \times V_n$,
\begin{align*}
\nu_0(x)&=\nu_{0}^n(x_n) \nu_{0}^{n-1}(x_{n-1}|x_n)\nu_{0}^{n-2}(x_{n-2}|x_{n-1},x_{n}) \cdots \nu_{0}^{1}(x_1|x_2,\ldots,x_n),\\
\nu_1(y)&=\nu_{1}^n(y_n) \nu_{1}^{n-1}(y_{n-1}|y_n)\nu_{1}^{n-2}(y_{n-2}|y_{n-1},y_{n}) \cdots \nu_{1}^{1}(y_1|y_2,\ldots,y_n).
\end{align*} 
For all $j=1, \ldots, n$, let $\pi^j(\,\cdot\, | x_{j+1},\ldots,x_{n},y_{j+1},\ldots,y_{n}) \in  \mathcal{P}(V_j^2)$ be a coupling
of $\nu_{0}^{j}(\,\cdot\,|x_{j+1},\dots,x_{n})$ and $\nu_{1}^{j}(\,\cdot\,|y_{j+1},\dots,y_{n})$. The Knothe-Rosenblatt coupling $\hat \pi \in  \mathcal{P}([V_1 \times \dots \times V_n]^2)$ between $\nu_{0}$ and $\nu_{1}$ is then defined by
$$
\hat \pi(x,y) = \pi^n(x_n,y_n) \pi^{n-1}(x_{n-1},y_{n-1}|x_n,y_n) \cdots  
 \pi^{1}(x_{1},y_{1}|x_2,\dots,x_n,y_2,\dots,y_n) ,
$$
for all $x=(x_{1},x_{2},\ldots,x_{n})$ and $y=(y_{1},y_{2},\ldots,y_{n}).$

\subsection{Tensorisation}
Another useful property of the weak transport cost defined above is that it tensorises in the following sense.
For $1\le i\le n$, let $G_i=(V_i,E_i)$ be a graph with the associated distance $d_i$.
Given two probability measures $\nu_0, \nu_1$ in $\mathcal{P}(V_1 \times \cdots\times V_{n})$, define
\begin{align*}
\Tw_{2}^{(n)}(\nu_1 |\nu_0 ) 
 := 
\inf_{p \in P(\nu_0,\nu_1)} 
\sum_{x \in V_1 \times\cdots \times V_{n}} \sum_{i=1}^n  \left( \sum_{y \in V_1 \times\cdots \times V_{n}} d_i(x_i,y_i) p(x,y) \right)^2 \nu_0(x)
\end{align*}
where $x=(x_1,\ldots,x_{n}), y=(y_1,\ldots,y_n) \in V_1 \times\cdots \times V_{n}$. 

As above, for any coupling $\pi$ of $\nu_0, \nu_1 \in \mathcal{P}(V_1 \times\cdots \times V_{n})$ we also define
$$
I_2^{(n)}(\pi) := \sum_{x \in V_1 \times\cdots \times V_{n}} \sum_{i=1}^n  \left( \sum_{y \in V_1 \times\cdots \times V_{n}} d_i(x_i,y_i) p(x,y) \right)^2 \nu_0(x)
$$
where $p$ is such that $\pi(x,y) = \nu_0(x)p(x,y)$, for all $x,y \in V_1 \times\cdots \times V_{n}$. Similarly, one defines $\bar{I}^{(n)}_{2}$.

We also define 
$$
J_2^{(n)}(\pi) := \sum_{i=1}^n  \left( \sum_{x,y \in V_1 \times\cdots \times V_{n}} d_i(x_i,y_i) \pi(x,y) \right)^2 
$$
and 
\begin{align*}
\hat {\mathcal{T}}_{2}^{(n)}(\nu_0 ,\nu_1 ) 
 := 
\inf_{\pi \in \Pi(\nu_0,\nu_1)} J_2^{(n)}(\pi).
\end{align*}

Using the notation of Section \ref{sec:knothe} above, we can state the result.

\begin{prop}\label{prop:tensor}
Let $\nu_0, \nu_1$ in $\mathcal{P}(V_1 \times \cdots\times V_n)$; and consider a family of couplings $\pi^n \in \Pi(\nu_{0}^n,\nu_{1}^n)$ and $\pi^k(\,\cdot\,| x_{k+1},\ldots,x_{n}) \in \Pi(\nu_{0}^k(\,\cdot\,| x_{k+1},\ldots,x_{n}) , \nu_{1}^k(\,\cdot\,| y_{k+1},\ldots,y_{n}))$ with $(x_{2},\ldots,x_{n}),(y_{2},\ldots,y_{n}) \in V_{2}\times\cdots\times V_{n}$, as above.
Then,
$$
I_2^{(n)}(\hat \pi) \leq I_2(\pi^n ) 
+ \sum_{k=1}^{n-1}\sum_{x,y \in V_1\times \cdots\times V_{n}} {\hat \pi}(x,y)   I_2(\pi^k(\,\cdot\,|x_{k+1},\ldots,x_{n},y_{k+1}\ldots y_{n})).
$$
where $\hat \pi$ is the Knothe-Rosenblatt coupling of $\nu_0$ and $\nu_1$ associated with the family of couplings above. The same holds for $\bar{I}_{2}^{(n)}$ and $ J_2^{(n)}(\pi)$.
\end{prop}
In particular, if the couplings $\pi^n$ and $\pi^k(\,\cdot\,| x_{k+1},\ldots,x_{n})$ are assumed to achieve the infimum in the definition of the weak transport costs between $\nu_{0}^n$ and $\nu_{1}^n$ and between $\nu_{0}^k(\,\cdot\,| x_{k+1},\ldots,x_{n})$ and $\nu_{1}^k(\,\cdot\,| y_{k+1},\ldots,y_{n})$ for all $k\in\{1,\ldots,n-1\}$, we immediately get the following tensorisation inequality for $\Tw_{2}$:
\begin{align}\label{eq:tensor}
\Tw_{2}^{(n)}(\nu_{1}|\nu_{0}) & \leq \Tw_{2}(\nu_{1}^n|\nu_{0}^n)
+ \sum_{k=1}^{n-1}\sum_{\genfrac{}{}{0pt}{}{x,y \in }{V_1\times \cdots\times V_{n}}} {\hat \pi}(x,y)   \Tw_{2}(\nu_{1}^k(\cdot| x_{k+1},\ldots,x_{n}) | \nu_{0}^k(\cdot| y_{k+1},\ldots,y_{n})).
\end{align}
In an obvious way, the same kind of conclusion holds replacing $\Tw_{2}$ by $\hat {\mathcal{T}}_{2}$.
\begin{proof}
In this proof, we will use the following shorthand notation: if $x\in V$ and if $1\leq i\leq j\leq n$, we will denote by $x_{i : j}$ the subvector $(x_{i},x_{i+1},\ldots,x_{j})\in V_i\times \cdots\times V_j.$

Define the kernels $\hat{p}(\,\cdot\,,\,\cdot\,)$, $p^n(\,\cdot\,,\,\cdot\,)$ and $p^{k}(\,\cdot\,,\,\cdot\, | x_{k+1 : n}, y_{k+1:n})$ by the formulas
\begin{align*}
\hat{\pi}(x,y)&= \hat{p}(x,y)\nu_{0}(x)\\
\pi^k(x_{k},y_{k} | x_{k+1:n}, y_{k+1:n})&=p^k(x_{k},y_{k} | x_{k+1:n},y_{k+1:n})\nu_{0}^k(x_{k}|x_{k+1:n}),\quad \forall k<n,\\
\pi^n(x_{n},y_{n})&=p^n(x_{n},y_{n})\nu_{0}^n(x_{n}).
\end{align*}
By the definition of the Knothe-Rosenblatt coupling $\hat{\pi}$, it holds
$$\hat{p}(x,y) = \prod_{k=1}^{n-1} p^k(x_{k},y_{k} |  x_{k+1:n},y_{k+1:n})\times p^n(x_{n},y_{n}) .$$
As a result,
\begin{align*}
& \left( \sum_{y } d_i(x_i,y_i)  \hat{p}(x,y) \right)^2
 =\left( \sum_{y_{i:n}} d_i(x_i,y_i)  \prod_{k=i}^{n-1} p^k(x_{k},y_{k} |  x_{k+1:n},y_{k+1:n})p^n(x_{n},y_{n})\right)^2\\
& \qquad \qquad \leq \sum_{y_{i+1:n}} \prod_{k=i+1}^{n-1} p^k(x_{k},y_{k} | x_{k+1:n}, y_{k+1:n})p^n(x_{n},y_{n}) \left(\sum_{y_{i}} d_i(x_i,y_i)  p^i(x_{i},y_{i} | x_{i+1:n},y_{i+1:n})\right)^2
\end{align*}
where the  inequality comes from Jensen's inequality.
Therefore,
\begin{align*}
&\sum_{x} \left( \sum_{y} d_i(x_i,y_i) \hat{p}(x,y) \right)^2\nu_{0}(x) \\
& \leq \sum_{x_{i+1:n}}\sum_{y_{i+1:n}}  \prod_{k=i+1}^{n-1} \pi^k(x_{k},y_{k} | x_{k+1:n},y_{k+1:n})\pi^n(x_{n},y_{n}) 
\sum_{x_{i}}\nu_{0}^{i}(x_{i}|x_{i+1:n})\left(\sum_{y_{i}} d_i(x_i,y_i)  p^i(x_{i},y_{i} |x_{i+1:n},y_{i+1:n})\right)^2\\
& \ =  \sum_{x_{i+1:n}}\sum_{y_{i+1:n}}  \prod_{k=i+1}^{n-1} \pi^k(x_{k},y_{k} | x_{k+1:n},y_{k+1:n})\pi^n(x_{n},y_{n}) I_{2}(\pi^{i}(\,\cdot\,|x_{i+1:n},y_{i+1:n})) \\
&  = \sum_{x,y} \hat{\pi}(x,y)I_{2}(\pi^{i}(\,\cdot\,|x_{i+1:n},y_{i+1:n})).
\end{align*}
Similarly
$$\sum_{x} \left( \sum_{y} d_n(x_n,y_n) \hat{p}(x,y) \right)^2\nu_{0}(x)\leq \sum_{x,y} \hat{\pi}(x,y)I_{2}(\pi^{n}).$$
Summing all these inequalities gives the announced tensorisation formula.

The proof for $\bar{I}_{2}^{(n)}$ and $J_2^{(n)}$ is identical and left to the reader.
\end{proof}

%As an illustration, we apply Proposition \ref{prop:tensor} to the specific example of the $n$-dimensional hypercube $\Omega_n$. On $\Omega_n$, the weak transport cost reads 
%\begin{equation} \label{weakhyper}
%\Tw_2^{(n)}(\nu_0|\nu_1) = \inf_{p \in P(\nu_0,\nu_1)} \sum_{x \in \Omega_n} \sum_{i=1}^n \left( \sum_{y \in \Omega_n} \1_{x_i \neq y_i} p(x,y) \right)^2 \nu_0(x) .
%\end{equation}
%Now if $\gamma = \frac{1}{2} \delta_0 + \frac{1}{2} \delta_1$ and $\nu_0=\gamma^n=\gamma \otimes \dots \otimes \gamma$
%is the $n$-fold product of $\gamma$ (\textit{i.e.}\ the uniform measure on $\Omega_n$), one has, thanks to Proposition \ref{prop:tensor},
%for all $i \in \{1,\dots,n\}$, (setting $\nu=\nu_1$ for simplicity)
%\begin{align*}
%\Tw_{2}^{(n)}(\gamma^n|\nu)
%&\leq 
%\Tw_{2}(\gamma|\nu_{i})+\sum_{x_i \in \{0 , 1\}} \Tw_2^{(n-1)}(\gamma^{n-1} |\nu(x_i,\cdot))\nu_{i}(x_i)\\
%&=  
%\Tw_{2}(\gamma|\nu_{i})+\nu_{i}(0) \Tw_2^{(n-1)}(\gamma^{n-1}|\nu(0,\cdot))+ \nu_{i}(1) \Tw_2^{(n-1)}(\gamma^{n-1}|\nu(1,\cdot))
%\end{align*}
%where  $\nu_i$ is the $i$-th marginal of $\nu$, \textit{i.e.}\ $\nu_i(x_i)=\nu(\{0,1\}^{i-1},x_i,\{0,1\}^{n-i-1})$,
%$x_i \in \{0,1\}$.

%%%%%%%%%%%%%%%%%%%%%%%%%%%%%%%%%%%%%%%%%%%%%%%%%%%%%%%%%%%%%%%%%%%%%%%%%%%%%%%%%%%%
%%%%%%%%%%%%%%%%%%%%%%%%%%%%%%%%%%%%%%%%%%%%%%%%%%%%%%%%%%%%%%%%%%%%%%%%%%%%%%%%%%%%
%%%%%%%%%%%%%%%%%%%%%%%%%%%%%%%%%%%%%%%%%%%%%%%%%%%%%%%%%%%%%%%%%%%%%%%%%%%%%%%%%%%%

\section{Displacement convexity property of the entropy.} \label{sec:Section DC}

Using the weak transport cost defined in the previous section, we can now derive a displacement convexity property of the entropy on graphs. 
More precisely, we will derive such a property for the complete graph. Then we will prove that  our definition of $\nu_t^\pi$ allows the displacement convexity to tensorise. As a consequence, we will be able to derive such a property on the $n$-dimensional hypercube.
%$\Omega_n$. 

\subsection{The complete graph}

Consider the complete graph $K_n$, or equivalently any graph $G$ equipped with the Hamming distance $d(x,y)=\1_{x \neq y}$ (in the definition of the weak transport cost). Recall the definition of $\nu_t^\pi$ given in \eqref{atlantabis}, and
that we proved, in Section \ref{sec:complete}, that $\nu_t^\pi=(1-t)\nu_0+t\nu_1$ for any choice of coupling $\pi$.
Then, the following holds.

\begin{prop}[Displacement convexity on the complete graph]\label{convexity}
Let $\nu_0$ ,$\nu_1$, $\mu \in \mathcal{P}(K_n)$ be three probability measures. Assume that $\nu_0,\nu_1$ are absolutely continuous with respect to $\mu$. Then
$$
H(\nu_t|\mu)\leq (1-t)H(\nu_0|\mu)+tH(\nu_1|\mu)-\frac{t(1-t)}{2}\left(\Tw_2(\nu_1|\nu_0)+\Tw_2(\nu_0|\nu_1)\right),\quad \forall t\in [0,1],
$$
where $\nu_t=(1-t)\nu_0+t\nu_1$.
\end{prop}

\begin{proof}
Our aim is simply to bound from below the second order derivative of $t \mapsto F(t):=H(\nu_t|\mu)$.
Denote by $f_0$ and $f_1$ the respective densities of $\nu_0$ and $\nu_1$ with respect to $\mu$. We have
$$
F(t)=\int \log\left((1-t)f_0+tf_1\right)\left((1-t)f_0+t f_1\right)\,d\mu.
$$
Thus $F'(t)=\int \log\left((1-t)f_0+tf_1\right)\,d(\nu_0-\nu_1)$. In turn,
\begin{align*}
F''(t)
&=
\int \frac{(f_0-f_1)^2}{(1-t)f_0+tf_1}\,d\mu 
=
\int \frac{[f_0-f_1]_+^2}{(1-t)f_0+tf_1}\,d\mu + \int \frac{[f_1-f_0]_+^2}{(1-t)f_0+tf_1}\,d\mu\\
&\geq 
\int \frac{[f_0-f_1]_+^2}{f_0}\,d\mu + \int \frac{[f_1-f_0]_+^2}{f_1}\,d\mu 
 =
\int \left[1-\frac{f_1}{f_0}\right]_+^2 f_0\,d\mu + \int \left[1-\frac{f_0}{f_1}\right]_+^2f_1\,d\mu \\
 & =
\Tw_2(\nu_1|\nu_0)+\Tw_2(\nu_0|\nu_1),
\end{align*}
where, in the last line, we used Lemma \ref{lem:postive}.
As a consequence, the function 
$
G \colon t\mapsto F(t)-\frac{t^2}{2}\left(\Tw_2(\nu_1|\nu_0)+\Tw_2(\nu_0|\nu_1)\right)
$ 
is convex on $[0,1],$ so that $G(t)\leq (1-t)G(0)+tG(1)$ which gives precisely, after some algebra,  the desired inequality.
\end{proof}

\begin{rem}[Pinsker inequality]
As an immediate consequence of the previous proposition, we will derive Csiszar-Kullback-Pinsker inequality (\cite{pinsker,kullback,csiszar}).
Recall the notation of the proof of Proposition \ref{convexity}.
Applying Cauchy-Schwarz yields
\begin{align*}
F''(t)&=\int \left(\frac{|f_0-f_1|}{\sqrt{(1-t)f_0+tf_1}}\right)^2\,d\mu\int \left(\sqrt{(1-t)f_0+tf_1}\right)^2\,d\mu
\geq \left(\int |f_0-f_1|\,d\mu\right)^2
=\|\nu_0-\nu_1\|_{TV}^2.
\end{align*}
Hence the map $G:t\mapsto F(t)-\frac{t^2}{2}\|\nu_0-\nu_1\|_{TV}^2$ is convex on $[0,1]$ so that
\begin{equation} \label{eq:csiszar}
H(\nu_t|\mu)\leq (1-t)H(\nu_0|\mu)+tH(\nu_1|\mu)-\frac{t(1-t)}{2}\|\nu_0-\nu_1\|_{TV}^2,\qquad \forall t\in [0,1] .
\end{equation}
Inequality \eqref{eq:csiszar} is a reinforcement of the well known Csiszar-Kullback-Pinsker's inequality (see e.g. \cite[Theorem 8.2.7]{ane}) which asserts that 
$$
\|\nu_0-\nu_1\|_{TV}^2 \leq 2 H(\nu_1|\nu_0) .
$$
%%PT -- I edited the sentence and moved it to after the equation. Also t goes to 0, not to 1.
%%
Indeed, take $\mu=\nu_0$ together with the fact that $H(\nu_t|\mu) \geq 0$, and then take the limit $t \to 0$ in \eqref{eq:csiszar} to obtain the above inequality.

Csiszar-Kullback-Pinsker's inequality, and its generalizations, are known to have many applications in Probability theory, Analysis and Information theory, see \cite[Page 636]{villani} for a review.

Now we compare the displacement convexity property of Proposition \ref{convexity} with \eqref{eq:csiszar}. For the two-point space it is easy to check that  the ratio 
$$
\frac{\Tw_2(\nu_1|\nu_0)+\Tw_2(\nu_0|\nu_1)}{\|\nu_0-\nu_1\|_{TV}^2}
$$
is not uniformly bounded above over all probability measures $\nu_0$ and $\nu_1$. On the other hand, we claim that
\begin{eqnarray}\label{bonrepas}
\frac{\Tw_2(\nu_1|\nu_0)+\Tw_2(\nu_0|\nu_1)}{\|\nu_0-\nu_1\|_{TV}^2}  \geq \frac{1}{2}\,, \qquad \forall \nu_0, \nu_1
\end{eqnarray}
%with $1/2$ the optimal constant, 
which implies that the result in Proposition~\ref{convexity} is 
%(strictly) 
stronger than \eqref{eq:csiszar}, up to a constant 2. %Notice however that we cannot exactly recover \eqref{eq:csiszar} since $1/2$ is optimal.
%%PT : I added the following sentence
%%
We also provide an example below which shows that  we cannot exactly recover \eqref{eq:csiszar} using Proposition~\ref{convexity}.

Let us prove the claim, and more precisely that the following holds
\begin{equation}\label{near-tightness}
\Tw_2(\nu_1|\nu_0)+\Tw_2(\nu_0|\nu_1) \geq \frac{\|\nu_0-\nu_1\|_{TV}^2}{1+\frac{\|\nu_0-\nu_1\|_{TV}}{2}} \geq\frac 12\|\nu_0-\nu_1\|_{TV}^2.
\end{equation}
This is a consequence of Cauchy-Schwarz inequality, namely, we have
$$
\Tw_2(\nu_1|\nu_0)+\Tw_2(\nu_0|\nu_1) 
\geq 
\frac{\left(\int [f_1-f_0]_+ d\mu\right)^2}{\nu_1(f_1\geq f_0)}+\frac{\left(\int [f_0-f_1]_+ d\mu\right)^2}{\nu_0(f_0> f_1)}.
$$
Since  $ \|\nu_0-\nu_1\|_{TV}= 2\int [f_1-f_0]_+ d\mu= 2(\nu_1(f_1\geq f_0)-\nu_0(f_1\geq f_0))$, we get 
$$
\Tw_2(\nu_1|\nu_0)+\Tw_2(\nu_0|\nu_1)
\geq  
\inf_{u\in [0,1]}\frac{(1+\frac{\|\nu_0-\nu_1\|_{TV}}{2})\|\nu_0-\nu_1\|_{TV}^2}{4u(1+\frac{\|\nu_0-\nu_1\|_{TV}}{2}-u)}  
= 
\frac{\|\nu_0-\nu_1\|_{TV}^2}{1+\frac{\|\nu_0-\nu_1\|_{TV}}{2}} .
$$
We now give the example that achieves {\em equality} in the first inequality of \eqref{near-tightness}, thus confirming that Proposition~\ref{convexity} can not exactly recover \eqref{eq:csiszar} :
%It remains to prove that $1/2$ is optimal. To that purpose,
Let $\nu_0$ and $\nu_1$ be two probability measures on the two-point space $\{0,1\}$ defined by  $\nu_1(1)=\nu_0(0)=3/4$ and $\nu_1(0)=\nu_0(1)=1/4$. Then 
$$
 \|\nu_0-\nu_1\|_{TV}=2(\nu_1(1)-\nu_0(1))=1,
 $$
 and 
 $$
 \Tw_2(\nu_1|\nu_0)+\Tw_2(\nu_0|\nu_1)= \frac{(\nu_1(1)-\nu_0(1))^2}{\nu_1(1)}+\frac{(\nu_0(0)-\nu_1(0))^2}{\nu_0(0)}=2/3,
 $$
 which gives the (claimed) equality in \eqref{near-tightness}.
\end{rem}

%%%%%%%%%%%%%%%%%%%%%%%%%%%%%%%%%%%%%%%%%%%%%%%%%%%%%%%%%%%%%%%%%%%%%%%%%%%%%%%%%%%%

\subsection{Tensorisation of the displacement convexity property}

In this section we prove that if the displacement convexity property of the entropy holds on $n$ graphs $G_1=(V_1,E_1)$, \ldots, $G_n=(V_n,E_n)$, equipped with probability measures $\mu_1,\ldots,\mu_n$ and graph distances $d_1,\ldots,d_n$ respectively, then the displacement convexity of the entropy holds on their Cartesian product equipped with $\mu_1\otimes\cdots\otimes\mu_n$ with respect to the tensorised transport costs $I_2^{(n)}$ and $\bar{I}_2^{(n)}$. As an application we shall apply such a property to the specific example of the hypercube at the end of the section. 

The next theorem is one of our main results.
\begin{thm} \label{th:main}
Let $(\mu^1,\ldots,\mu^n) \in \mathcal{P}(V_1) \times\cdots\times \mathcal{P}(V_n)$.
Assume that for all $i\in \{1,\ldots,n\}$ there is a constant $C_i\geq 0$ such that for all $\nu_0,\nu_1 \in \mathcal{P}(V_i)$ there exists $\pi = \pi^i\in \Pi(\nu_0,\nu_1)$ such that for all $t\in [0,1]$ it holds that:
$$H(\nu_t^\pi | \mu^i) \leq (1-t)H(\nu_0 | \mu^{i}) + tH(\nu_1 | \mu^{i}) - C_it(1-t)(I_2(\pi)+\bar{I}_2(\pi)).$$
Then the product probability measure $\mu=\mu^1\otimes\cdots\otimes\mu^n$ defined on $G=(V,E)=G_1\boxempty\cdots\boxempty G_n$ verifies the following property: 
for all $\nu_0,\nu_1 \in \mathcal{P}(V)$ there exists $\pi = \pi^{(n)}\in \Pi(\nu_0,\nu_1)$ such that for all $t\in [0,1]$ it holds that:
$$H(\nu_t^\pi | \mu) \leq (1-t)H(\nu_0 | \mu) + tH(\nu_1 | \mu) - Ct(1-t)(I_2^{(n)}(\pi)+\bar{I}^{(n)}_2(\pi)),$$
where $C=\min_i C_i.$
The same proposition holds replacing $I_2(\pi)+\bar{I}_2(\pi)$ by $J_2(\pi)$ and $I_2^{(n)}(\pi)+\bar{I}^{(n)}_2(\pi)$ by $J_2^{(n)}(\pi)$.
\end{thm}

\begin{proof}
In this proof, we use the notation and definitions introduced in Section \ref{sec:knothe}.
Fix $\nu_0,\nu_1 \in  \mathcal{P}(V)$ and write the following disintegration formulas 
\begin{align*}
\nu_0(x)&=\nu_0^n(x_n)\prod_{i=1}^{n-1} \nu_0^k(x_k | x_{k+1 : n}),\qquad \forall x=(x_1,\ldots,x_n)\in V\\
\nu_1(y)&=\nu_1^n(y_n)\prod_{i=1}^{n-1} \nu_1^k(y_k | y_{k+1 : n}),\qquad \forall y=(y_1,\ldots,y_n)\in V,
\end{align*}
where we recall that $x_{k+1:n}=(x_{k+1},\ldots,x_{n})\in V_{k+1}\times \cdots\times V_{n}.$

By assumption, for every $x,y \in V$, there are couplings $\pi^n\in \mathcal{P}(V_n\times V_n)$ and $\pi^k(\,\cdot\, | x_{k+1:n},y_{k+1:n}) \in \mathcal{P}(V_k\times V_k)$ such that 
$$\pi^n\in\Pi(\nu_0^n,\nu_1^n)\quad\text{and}\quad \pi^k(\,\cdot\, | x_{k+1:n},y_{k+1:n})\in\Pi(\nu_0^k(\,\cdot\, |x_{k+1:n}),\nu_1^k(\,\cdot\, |y_{k+1:n})),$$
and for which the following inequalities hold
\begin{align*}
H(\nu_t^{n}|\mu^n)
&\leq 
(1-t)H(\nu_0^n | \mu^n) + tH(\nu_1^n | \mu^n) - C_nt(1-t)J_2(\pi^n),\\
H(\nu_t^{k, x_{k+1:n},y_{k+1:n}}|\mu^k)
&\leq (1-t)H(\nu_0^k(\,\cdot\, |x_{k+1:n}) | \mu^k) + tH(\nu_1^k(\,\cdot\, | y_{k+1:n}) | \mu^k) \\
&\quad
- C_kt(1-t)R_2(\pi^k(\,\cdot\,|x_{k+1:n}, y_{k+1:n})),
\end{align*}
where $R_2:=I_2+\bar{I}_2$, $\nu_t^{n}:=\nu_t^{\pi_n}$, and $\nu_t^{k, x_{k+1:n},y_{k+1:n}}=\nu_t^{\pi^k(\,\cdot\, | x_{k+1:n},y_{k+1:n})}.$

Now, consider the Knothe-Rosenblatt coupling $\hat{\pi}\in \Pi(\nu_0,\nu_1)$ constructed from the couplings $\pi^n$ and $\pi^k(\,\cdot\,| x_{k+1:n}, y_{k+1:n}),$ $x,y\in V$ and denote by $\gamma_t$ the path $\nu_t^{\hat{\pi}}\in \mathcal{P}(V)$ connecting $\nu_0$ to $\nu_1.$

Let us consider the disintegration of $\gamma_t$ with respect to its marginals:
$$\gamma_t(z)=\gamma_t^n(z_n)\gamma_t^{n-1}(z_{n-1}|z_n)\cdots \gamma_t^{1}(z_1|z_2,\ldots,z_n).$$
We claim that there exist non-negative coefficients $\alpha_{t}^{k}(x_{k+1:n},y_{k+1:n},z_{k+1:n})$ such that $$\sum_{x_{k+1:n},y_{k+1:n}}\alpha_{t}^k(x_{k+1:n},y_{k+1:n},z_{k+1:n})=1$$ and such that for all $k\in \{1,\ldots,n-1\}$ it holds
$$\gamma_{t}^k(\,\cdot\,|z_{k+1:n}) = \sum_{x_{k+1:n}, y_{k+1:n}} \nu_{t}^{k,x_{k+1:n}, y_{k+1:n}}(\,\cdot\,)\alpha_{t}^{k}(x_{k+1:n},y_{k+1:n},z_{k+1:n}).$$
Indeed, by definition and using the tensorisation property of $\nu_t^{x,y}$ given in Lemma \ref{tensor}, it holds
\begin{align*}
\gamma_t(z)&=\sum_{x,y\in V} \nu_t^{x,y}(z)\hat{\pi}(x,y).
\end{align*}
So, using the fact that, according to Lemma \ref{tensor},  $\nu_t^{x,y}(z)=\prod_{i=1}^n\nu_t^{x_i,y_i}(z_i)$, we see that
\begin{align*}
\sum_{u \in V : u_{k:n}=z_{k:n}} \gamma_t(u) &= \sum_{x,y\in V}\left(\sum_{u \in V : u_{k:n}=z_{k:n}}\nu_t^{x,y}(u)\right)\hat{\pi}(x,y)
=\sum_{x,y\in V} \prod_{i=k}^n \nu_t^{x_i,y_i}(z_i)\hat{\pi}(x,y)\\
&=\sum_{x_{k:n},y_{k:n}} \prod_{i=k}^n \nu_t^{x_i,y_i}(z_i)\pi^i(x_i,y_i |x_{i+1:n}, y_{i+1:n})\\ 
&=\sum_{x_{k+1:n},y_{k+1:n}} \nu_t^{k, x_{k+1:n}, y_{k+1:n}}(z_k)\prod_{i=k+1}^n \nu_t^{x_i,y_i}(z_i)\pi^i(x_i,y_i |x_{i+1:n}, y_{i+1:n}) .
\end{align*}
From this it follows that
\begin{align*}
\gamma_t^k(z_k | z_{k+1:n})&= \frac{\displaystyle \sum_{u \in V : u_{k:n}=z_{k:n}} \gamma_t(u)}{\displaystyle\sum_{u \in V : u_{k+1:n}=z_{k+1:n}}\gamma_t(u)}
= \frac{\displaystyle\sum_{x_{k+1:n},y_{k+1:n}} \nu_t^{k, x_{k+1:n}, y_{k+1:n}}(z_k)\prod_{i=k+1}^n \nu_t^{x_i,y_i}(z_i)\pi^i(x_i,y_i |x_{i+1:n}, y_{i+1:n})}{\displaystyle\sum_{x_{k+1:n},y_{k+1:n}} \prod_{i=k+1}^n \nu_t^{x_i,y_i}(z_i)\pi^i(x_i,y_i |x_{i+1:n}, y_{i+1:n})}\\
&:= \sum_{x_{k+1:n},y_{k+1:n}} \nu_t^{k, x_{k+1:n}, y_{k+1:n}}(z_k) \alpha_{t}^k(x_{k+1:n}, y_{k+1:n}, z_{k+1:n}),
\end{align*}
using obvious notation, from which the claim follows.
Similarly, for all $z_n\in V_n$, it holds
$\gamma^n_t(z_n)=\nu^n_t(z_n).$

Now, let us recall the well known disintegration formula for the relative entropy: if $\gamma \in \mathcal{P}(V)$ is absolutely continuous with respect to $\mu$, then it holds
\begin{equation}\label{disint-ent}
H(\gamma|\mu) = H(\gamma^n | \mu^n) + \sum_{k=1}^{n-1} \sum_{z\in V}  H(\gamma^{k}(\,\cdot\,| z_{k+1:n}) | \mu^{k}) \gamma(z).
\end{equation}
Applying \eqref{disint-ent} to $\gamma_{t}$, and the (classical) convexity of the relative entropy, it holds
\begin{align*}
H(\gamma_{t}|\mu)
&= 
H(\gamma_{t}^n | \mu^n) + \sum_{k=1}^{n-1} \sum_{z\in V}  H(\gamma_{t}^{k}(\,\cdot\,| z_{k+1:n}) | \mu^{k}) \gamma_{t}(z)\\
&\leq 
H(\nu_{t}^n |\mu^n) +\sum_{k=1}^{n-1} \sum_{z\in V}
\sum_{\genfrac{}{}{0pt}{}{x_{k+1:n},}{ y_{k+1:n}}} 
\alpha_{t}^k(x_{k+1:n},y_{k+1:n},z_{k+1:n}) H(\nu_{t}^{k, x_{k+1:n},y_{k+1:n}} | \mu^{k}) \gamma_{t}(z)
\end{align*}
Now we deal with each term in the sum separately. Fix $k \in \{1,\dots,n-1\}$. We have
\begin{align*}
& \sum_{z\in V}
\sum_{\genfrac{}{}{0pt}{}{x_{k+1:n},}{ y_{k+1:n}}} 
\alpha_{t}^k(x_{k+1:n},y_{k+1:n},z_{k+1:n}) H(\nu_{t}^{k, x_{k+1:n},y_{k+1:n}} | \mu^{k}) \gamma_{t}(z) \\
&= 
\sum_{z_{k+1:n}} \sum_{\genfrac{}{}{0pt}{}{x_{k+1:n},}{ y_{k+1:n}}} \alpha_{t}^k(x_{k+1:n},y_{k+1:n},z_{k+1:n})  
H(\nu_{t}^{k, x_{k+1:n},y_{k+1:n}} | \mu^{k}) 
 \sum_{\genfrac{}{}{0pt}{}{u\in V :}{u_{k+1:n}=z_{k+1:n}}}\gamma_{t}(u)\\
& = 
\sum_{z_{k+1:n}}\sum_{\genfrac{}{}{0pt}{}{x_{k+1:n},}{ y_{k+1:n}}} H(\nu_{t}^{k, x_{k+1:n},y_{k+1:n}} | \mu^{k}) \prod_{i=k+1}^n \nu_{t}^{x_{i},y_{i}}(z_{i})\pi^i(x_{i},y_{i} | x_{i+1:n} y_{i+1:n})\\
& = 
\sum_{\genfrac{}{}{0pt}{}{x_{k+1:n},}{ y_{k+1:n}}} H(\nu_{t}^{k, x_{k+1:n},y_{k+1:n}} | \mu^{k}) \prod_{i=k+1}^n\pi^i(x_{i},y_{i} | x_{i+1:n} y_{i+1:n}) \\
& =
\sum_{x, y} H(\nu_{t}^{k, x_{k+1:n},y_{k+1:n}} | \mu^{k}) \prod_{i=1}^n\pi^i(x_{i},y_{i} | x_{i+1:n} y_{i+1:n})\,.
\end{align*}
Therefore,
\begin{align*}
H(\gamma_{t}|\mu) 
& \leq H(\nu_{t}^n |\mu^n) +\sum_{k=1}^{n-1} \sum_{x, y} H(\nu_{t}^{k, x_{k+1:n},y_{k+1:n}} | \mu^{k}) \hat{\pi}(x,y).
\end{align*}
Now, applying the assumed displacement convexity inequalities, we get
\begin{align*}
H(\gamma_{t}|\mu) &\leq (1-t) \left[H(\nu_{0}^n |\mu^n) +\sum_{k=1}^{n-1} \sum_{x, y} H(\nu_{0}^{k}(\,\cdot\,|x_{k+1:n}) | \mu^{k}) \hat{\pi}(x,y) \right] \\
&\quad+ t\left[H(\nu_{1}^n | \mu^n) +\sum_{k=1}^{n-1} \sum_{x, y} H(\nu_{1}^{k}(\,\cdot\,| y_{k+1:n}) | \mu^{k}) \hat{\pi}(x,y)\right]\\
&\quad - Ct(1-t)\left[J_{2}(\pi^n) + \sum_{k=1}^{n-1} \sum_{x,y}R_{2}(\pi^k(\,\cdot\,| x_{k+1:n}, y_{k+1:n}))\hat{\pi}(x,y)\right]\\
&= (1-t) \left[H(\nu_{0}^n |\mu^n) +\sum_{k=1}^{n-1} \sum_{x} H(\nu_{0}^{k}(\,\cdot\,|x_{k+1:n}) | \mu^{k})\nu_{0}(x) \right] \\
&\quad+ t\left[H(\nu_{1}^n | \mu^n) +\sum_{k=1}^{n-1} \sum_{y} H(\nu_{1}^{k}(\,\cdot\,| y_{k+1:n}) | \mu^{k}) \nu_{1}(y)\right]\\
&\quad - Ct(1-t)\left[J_{2}(\pi^n) + \sum_{k=1}^{n-1}\sum_{x,y} R_{2}(\pi^k(\,\cdot\,| x_{k+1:n}, y_{k+1:n}))\hat{\pi}(x,y)\right]\\
&\leq (1-t)H(\nu_{0}|\mu) + tH(\nu_{1}|\mu) - Ct(1-t) (I_{2}^{(n)}(\hat{\pi})+\bar{I}_{2}^{(n)}(\hat{\pi})),
\end{align*}
where the last inequality follows from the disintegration equality \eqref{disint-ent} for the relative entropy and from the disintegration inequality given in Proposition \ref{prop:tensor}.
\end{proof}

As an application of Theorem \ref{th:main}, we derive the displacement convexity of  entropy property on the hypercube.
%$\Omega_n$.

\begin{cor}[Displacement convexity on the hypercube] \label{cor:dchypercube}
Let $\mu$ be a probability measure on $\{0,1\}$ and define its $n$-fold product $\mu^{\otimes n}$ on $\Omega_n = \{0,1\}^n$.
For any $\nu_0, \nu_1 \in \mathcal{P}(\Omega_n)$, there exists a $\pi \in \Pi(\nu_{0},\nu_{1})$ such that for any $t \in [0,1]$, 
\begin{align} \label{pluie}
H(\nu_t^{\pi}|\mu^{\otimes n}) \leq 
(1-t)H(\nu_0|\mu^{\otimes n})+tH(\nu_1|\mu^{\otimes n}) 
-\frac{t(1-t)}{2}\left(I_2^{(n)}(\pi)+ \bar{I}_{2}^{(n)}(\pi) \right).
\end{align} 
and there exists $\pi \in \Pi(\nu_{0},\nu_{1})$ such that for any $t \in [0,1]$, 
\begin{align} \label{pluies}
H(\nu_t^{\pi}|\mu^{\otimes n}) \leq 
(1-t)H(\nu_0|\mu^{\otimes n})+tH(\nu_1|\mu^{\otimes n}) 
-{2t(1-t)}J_2^{(n)}(\pi).
\end{align} 

\end{cor}
\proof
According to Proposition \ref{convexity}, for all $\nu_0,\nu_1 \in \mathcal{P}(\{0,1\})$, it holds
$$H(\nu_t|\mu) \leq 
(1-t)H(\nu_0|\mu)+tH(\nu_1|\mu) 
-\frac{t(1-t)}{2}\left(\Tw_{2}(\nu_1|\nu_0)+\Tw_{2}(\nu_0|\nu_1) \right),\qquad \forall t\in [0,1],$$
with $\nu_t =(1-t)\nu_0 + t\nu_1.$ It is not difficult to check that the coupling $\pi$ defined by \eqref{couplagebis} is optimal for both $\Tw_{2}(\nu_1|\nu_0)$ and $\Tw_{2}(\nu_0|\nu_1)$. Since on the two-point space $\nu_t=\nu_t^\pi$ is independent of $\pi$, the preceding inequality can be rewritten as follows:
$$H(\nu_t^\pi|\mu) \leq 
(1-t)H(\nu_0|\mu)+tH(\nu_1|\mu) 
-\frac{t(1-t)}{2}\left(I_2(\pi)+\bar{I}_2(\pi) \right),\qquad \forall t\in [0,1].$$
Therefore, we are in a position to apply Theorem \ref{th:main}, and to conclude that $\mu^{\otimes n}$ verifies the announced displacement convexity property \eqref{pluie}.

Similarly, by Lemma \ref{lem:postive}, the displacement convexity property \eqref{eq:csiszar} ensures that 
$$H(\nu_t^\pi|\mu) \leq 
(1-t)H(\nu_0|\mu)+tH(\nu_1|\mu) 
-{2t(1-t)}J_2(\pi),\qquad \forall t\in [0,1].$$
The result then follows from Theorem \ref{th:main}.
\endproof

Let $\pi$ be a coupling of $\nu_0,\nu_1 \in \mathcal{P}(\Omega_n)$.  By the Cauchy-Schwarz inequality, we have
\begin{align*}
J_2^{(n)}(\pi) 
& = 
\sum_{i=1}^n \left( \sum_{x,y \in \Omega_n} \1_{x_i \neq y_i} \pi(x,y) \right)^2  \geq 
\frac{1}{n} \left( \sum_{x,y \in \Omega_n} \sum_{i=1}^n \1_{x_i \neq y_i} \pi(x,y)  \right)^2 
 =
\frac{1}{n} \left( \sum_{x,y \in \Omega_n} d(x,y)  \pi(x,y) \right)^2 \\
& \geq \frac{1}{n} W_1 (\nu_1,\nu_0)^2 .
\end{align*}
We immediately deduce from Corollary  \ref{cor:dchypercube} the following weaker result.
\begin{cor}
Let $\mu$ be a probability measure on $\{0,1\}$ and define its $n$-fold product $\mu^{\otimes n}$ on $\Omega_n = \{0,1\}^n$. For any $\nu_0,\nu_1 \in \mathcal{P}(\Omega_n)$, there exists $\pi \in \Pi(\nu_0,\nu_1)$ such that for  $t \in [0,1]$,
\begin{align*} 
H(\nu_t^{\pi}|\mu^{\otimes n}) \leq 
(1-t)H(\nu_0|\mu^{\otimes n})+tH(\nu_1|\mu^{\otimes n}) 
-\frac{2t(1-t)}{n} W_1 (\nu_1,\nu_0)^2 .
\end{align*}
\end{cor}
The constant $1/n$ encodes, in some sense, the discrete Ricci curvature of the hypercube
in accordance with the various definitions of the discrete Ricci curvature (see the introduction).

\begin{rem}\label{remarcable}
Since $\Tw_2$ is defined as an infimum, one can replace, for free, the term $I_2^{(n)}(\pi)$
by $\Tw_2^{(n)}(\nu_1|\nu_0)$ in \eqref{pluie}. Moreover, if one chooses $\nu_0=\mu^{\otimes n}$ and uses that $H(\nu_t^{\pi}|\mu^{\otimes n}) \geq 0$, one easily derives from 
\eqref{pluie} the following transport-entropy inequality:
$$
\Tw_2^{(n)}(\nu|\mu^{\otimes n})+\Tw_2^{(n)}(\mu^{\otimes n}|\nu) \leq 2  H(\nu|\mu^{\otimes n}), \qquad \forall \nu \in \mathcal{P}(\Omega_n) .
$$
See \cite{GRST} for more on such an inequality (on graphs). Note that the above argument is general and that one can always derive from the displacement convexity of the entropy some Talagrand-type transport-entropy inequality.
\end{rem}

%%%%%%%%%%%%%%%%%%%%%%%%%%%%%%%%%%%%%%%%%%%%%%%%%%%%%%%%%%%%%%%%%%%%%%%%%%%%%%%%%%%%
%%%%%%%%%%%%%%%%%%%%%%%%%%%%%%%%%%%%%%%%%%%%%%%%%%%%%%%%%%%%%%%%%%%%%%%%%%%%%%%%%%%%
%%%%%%%%%%%%%%%%%%%%%%%%%%%%%%%%%%%%%%%%%%%%%%%%%%%%%%%%%%%%%%%%%%%%%%%%%%%%%%%%%%%%

\section{HWI type inequalities on graphs. }\label{sec:hwi}

As already stated in the introduction, the displacement convexity of entropy property is usually (\textit{i.e.,}\ in continuous space settings)  the strongest property in the following hierarchy:
$$
\mbox{Displacement convexity } \Rightarrow \mbox{ HWI }  \Rightarrow \mbox{Log Sobolev} . 
$$
Applying an argument based on the differentiation property of Corollary \ref{parisbis}, in this section, we derive HWI and log-Sobolev type inequalities from the displacement convexity property.

We shall start with a general statement on product of graphs that allows to obtain  symmetric  HWI inequality from the displacement convexity property of the entropy.  As a consequence, we get a new symmetric  HWI inequality on the hypercube that  implies a modified log-Sobolev inequality on the hypercube. This modified log-Sobolev inequality also implies, by means of the Central Limit Theorem, the classical log-Sobolev inequality for the standard Gaussian measure, with the optimal constant.

Then we move to another HWI type inequality involving the already mentioned Dirichlet form ${\mathcal{E}_\mu(f,\log f)}$ based on  Equation~\eqref{crepe} available on complete graph. 

%%%%%%%%%%%%%%%%%%%%%%%%%%%%%%%%%%%%%%%%%%%%%%%%%%%%%%%%%%%%%%%%%%%%%%%%%%%%%%%%%%%%

\subsection{Symmetric HWI inequality for products of graphs}

The main result of this section is the following abstract symmetric HWI inequality valid on the $n$-fold  product of any graph.

\begin{prop}[HWI]\label{HWIprop}
Consider $G^n$ for $G=(V,E)$ any graph and $\mu\in \mathcal{P}(V^n)$. Assume that $\mu$ verifies the following displacement convexity inequality: there is some $c>0$ such that for any $\nu_0, \nu_1 \in \mathcal{P}(V^n)$, there exists a coupling $\pi \in \Pi(\nu_0,\nu_1)$ such that
$$
H(\nu_t^\pi |\mu) \leq (1-t) H(\nu_0|\mu) + t H(\nu_1|\mu) -ct(1-t) (I_2^{(n)} (\pi)+\bar{I}_2^{(n)}(\pi)) \qquad \forall t \in [0,1].
$$
Then $\mu$ verifies 
\begin{align} \label{shwi}
H(\nu_0|\mu) 
&\leq 
H(\nu_1|\mu) + \sqrt{\sum_{x\in V^n} \sum_{i=1}^n \left[\sum_{z\in N_i(x)} \left( \log  \frac{\nu_0(x)}{\mu(x)} - \log \frac{\nu_0(z)}{\mu(z)}  \right)\right]_{+}^2\nu_0(x)}\sqrt{I_2^{(n)}(\pi)} %\nonumber \\ & \quad 
- c(I_2^{(n)}(\pi) + \bar{I}_2^{(n)}(\pi)),
\end{align}
for the same $\pi\in \Pi(\nu_0,\nu_1)$ as above, where $N_i(x)=\{z \in V^n ; d(x,z)=1\text{ and } x_i\neq z_i\}$.
\end{prop}

The proof of this result is given below.
Before proving that, we derive a certain {\em reinforced} log-Sobolev inequality (see below for a brief justification of the name) in the discrete setting, and as a consequence, the classical 
Gross' log-Sobolev inequality on the continuous line, with the optimal constant.

Choose $\nu_1=\mu$ in \eqref{shwi} and denote by $f(x)=\nu_0(x)/\mu(x)$. Then, using the elementary inequality $\sqrt{ab}\leq a/(2\e) + \e b/2$, $\e>0$, we immediately get the following corollary.
 
\begin{cor}[Reinforced log-Sobolev]\label{cortarte}
Under the same assumptions of Proposition \ref{HWIprop}, for all 
$f \colon V^n \to (0,\infty)$ with $\mu(f)=1$,  for all $\e \leq 2c$, it holds that
\begin{align} \label{hwihc}
\ent_\mu(f)
 \leq 
\frac{1}{2\e} \sum_{x\in V^n} \sum_{i=1}^n \left[\sum_{z\in N_i(x)} \left(\log  f(x) - \log f(z)  \right)\right]_{+}^2f(x)\mu(x) %\nonumber \\
%& \quad 
-(c-\frac{\e}{2}) \Tw_2(\mu|f\mu)-c \Tw_2(f\mu|\mu).
\end{align}
\end{cor}

Inequality \eqref{hwihc} can be seen as a reinforcement of a (discrete) modified log-Sobolev inequality.
The next corollary deals with the special case of the discrete cube.

\begin{cor}[Reinforced log-Sobolev on $\Omega_n$ and Gross' Inequality] \label{cor:lsob}
Let $\mu$ be a Bernoulli measure on $\{0,1\}$. Then, for any $n$ and any $f \colon \Omega_n \to (0,\infty)$, it holds
\begin{equation} \label{chaud}
\ent_{\mu^{\otimes n}}(f)\leq \frac{1}{2} \sum_{x\in \Omega_n} \sum_{i=1}^n  \left[ \log  f(x) - \log f(\sigma_i(x))  \right]_{+}^2f(x)\mu^{\otimes n}(x) -  \frac{1}{2} \Tw_2(f\mu|\mu)\,,
\end{equation}
where $\sigma_i(x)=(x_1,\dots,x_{i-1},1-x_i,x_{i+1},\dots,x_n)$ is the neighbor of $x=(x_1,\dots,x_n)$ for which the $ i$-th coordinate differs from that of $x$.

As a consequence, for any $n$ and any $g \colon \mathbb{R}^n \to \mathbb{R}$ smooth enough, it holds
\begin{equation} \label{gross}
\ent_{\gamma_n}(e^g)\leq  \frac{1}{2} \int |\nabla g|^2 e^g d\gamma_n
\end{equation}
where $\gamma_n$ is the standard Gaussian measure on $\mathbb{R}^n$, and $|\nabla g|$ is the length of the gradient of $g$.
\end{cor}

\begin{rem}
Note that the constant $1/2$ in the above log-Sobolev inequality for the standard Gaussian is optimal, see \textit{e.g.}\ \cite[Chapter 1]{ane}.
\end{rem}

%%In light of the two-point example in Section 2 and remarks there, the following is now unnecessary!
%%
%Due to the negative part, Inequality \eqref{chaud} is stronger than the following modified Logarithmic Sobolev inequality on the discrete cube:
%\begin{eqnarray}\label{mod+logsob}
%$$
%\ent_{\mu^{\otimes n}}(f)\leq  \frac{1}{2} \sum_{x\in \Omega_n} \sum_{i=1}^n  \left[ \log  f(x) - \log f(\sigma_i(x))  \right]_{+}^2f(x)\mu^{\otimes n}(x) ,
%$$
%\end{eqnarray}
%and hence the name, ``Reinforced log-Sobolev".

We proceed with the proofs of Proposition \ref{HWIprop} and Corollary \ref{cor:lsob}.

\begin{proof}[Proof of Proposition \ref{HWIprop}] The displacement convexity inequality ensures that for all $t \in [0,1]$, 
$$H(\nu_0|\mu)\leq H(\nu_1|\mu)- \frac{H(\nu_t  | \mu) - H(\nu_0|\mu)}{t} -c (I_2^{(n)} (\pi)+\bar{I}_2^{(n)}(\pi)).$$
As $t$ goes to 0, this yields
$$H(\nu_0|\mu)\leq H(\nu_1|\mu) - \frac{\partial}{\partial t}H(\nu_t^\pi|\mu)_{|t=0}- c(I_2^{(n)} (\pi)+\bar{I}_2^{(n)}(\pi)),$$
where $\pi \in\Pi(\nu_0,\nu_1)$.
According to Corollary \ref{parisbis}, it holds
\begin{align*}
- \frac{\partial}{\partial t}  H(\nu_t^\pi|\mu)_{|t=0} & =\sum_{\genfrac{}{}{0pt}{}{x,z \in V^n:}{z \sim x}} \left( \log  \frac{\nu_0(x)}{\mu(x)} - \log \frac{\nu_0(z)}{\mu(z)}  \right)
\sum_{y \in V^n} d(x,y)\frac{|\Gamma(x,z,y)|}{|\Gamma(x,y)|} \pi(x,y)\\
&= \sum_{\genfrac{}{}{0pt}{}{x,z \in V^n:}{z \sim x}} \left( \log  \frac{\nu_0(x)}{\mu(x)} - \log \frac{\nu_0(z)}{\mu(z)}  \right)
\sum_{y \in V^n} d(x,y)\frac{|\Gamma(x,z,y)|}{|\Gamma(x,y)|} \pi(x,y)\\
&\leq \sum_{x\in V^n} \sum_{i=1}^n \left[\sum_{z\in N_i(x)} \left( \log  \frac{\nu_0(x)}{\mu(x)} - \log \frac{\nu_0(z)}{\mu(z)} \right) \right]_{+}
\sum_{y \in V^n} d(x,y)\frac{|\Gamma(x,z,y)|}{|\Gamma(x,y)|} \pi(x,y) .
\end{align*}
According to \eqref{eq:geod-prod}, by induction on $n\ge 1$, we get  that for all $u,y\in V^n$, 
$$|\Gamma(u,y)|=\frac{d(u,y)!}{\prod_{j=1}^n d(u_j,y_j)!} \prod_{j=1}^n |\Gamma(u_j,y_j)|.$$
Applying this formula with $u=z\in N_i(x)$ for some $i\in \{1,\ldots,n\}$ and $u=x$, we get that for all $y$ such that $z\in \llbracket x,y\rrbracket$, it holds
\begin{equation}\label{changement}
\frac{|\Gamma(x,z,y)|}{|\Gamma(x,y)|}=\frac{|\Gamma(z,y)|}{|\Gamma(x,y)|}=\frac{d(z,y)!}{d(x,y)!} \frac{d(x_i,y_i)!}{d(z_i,y_i)!}\frac{|\Gamma(z_i,y_i)|}{|\Gamma(x_i,y_i)|}=\frac{d(x_i,y_i)}{d(x,y)}\frac{|\Gamma(z_i,y_i)|}{|\Gamma(x_i,y_i)|},
\end{equation}
using that $x_j=z_j$ for all $i\neq j$ and the relations $d(x,y)=1+d(z,y)$ and $d(x_i,y_i)=1+d(z_i,y_i).$
Therefore, when $z\in N_i(x)$,
$$\sum_{y \in V^n} d(x,y)\frac{|\Gamma(x,z,y)|}{|\Gamma(x,y)|} \pi(x,y)=\sum_{y \in V}  d(x_i,y_i)\frac{|\Gamma(x_i,z_i,y_i)|}{|\Gamma(x_i,y_i)|} \pi(x,y)\leq \sum_{y \in V}  d(x_i,y_i) \pi(x,y).$$
Plugging this inequality into the expression for  $-\frac{\partial}{\partial t}H(\nu_t^\pi|\mu)_{|t=0}$ yields:
\begin{align*}
-\frac{\partial}{\partial t}H(\nu_t^\pi|\mu)_{|t=0}
&\leq \sum_{x\in V^n} \sum_{i=1}^n \left[ \sum_{z\in N_i(x)} \left(\log  \frac{\nu_0(x)}{\mu(x)} - \log \frac{\nu_0(z)}{\mu(z)} \right) \right]_{+}
\sum_{y \in V^n} d(x_i,y_i)\pi(x,y)\\
&\leq \sum_{x\in V^n} \sum_{i=1}^n\left[ \sum_{z\in N_i(x)} \left( \log  \frac{\nu_0(x)}{\mu(x)} - \log \frac{\nu_0(z)}{\mu(z)} \right) \right]_{+}
\sum_{y \in V^n} d(x_i,y_i)\frac{\pi(x,y)}{\nu_0(x)}\,\nu_0(x)\\
&\leq \sqrt{\sum_{x\in V^n} \sum_{i=1}^n \left[\sum_{z\in N_i(x)} \left( \log  \frac{\nu_0(x)}{\mu(x)} - \log \frac{\nu_0(z)}{\mu(z)}  \right)\right]_{+}^2\nu_0(x)}\sqrt{I_2^{(n)}(\pi)},
\end{align*}
where the last line follows from the Cauchy-Schwarz inequality. This completes the proof.
\end{proof}

\begin{proof}[Proof of Corollary \ref{cor:lsob}]
By Corollary \ref{cor:dchypercube}, 
Inequality \eqref{hwihc} holds with $c=1/2$. Observe that $N_i(x)=\{\sigma_i(x)\}$ where 
$\sigma_i(x)=(x_1,\dots,x_{i-1},1-x_i,x_{i+1},\dots,x_n)$ is the neighbor of $x=(x_1,\dots,x_n)$ for which the $ i$-th coordinate differs from that of $x$. For $\e=1$, Corollary~\ref{cortarte} gives 
\begin{eqnarray*}
\ent_{\mu^{\otimes n}}(f)\leq \frac{1}{2} \sum_{x\in \Omega_n} \sum_{i=1}^n  \left[ \log  f(x) - \log f(\sigma_i(x))  \right]_{+}^2f(x)\mu^{\otimes n}(x) -  \frac{1}{2} \Tw_2(f\mu|\mu)\,,
\end{eqnarray*}
which is the first part of the corollary. 

For the second part, we shall apply the Central Limit Theorem. Our starting point is the 
following modified log-Sobolev inequality on the hypercube:
\begin{eqnarray}\label{mod+logsob}
\ent_{\mu^{\otimes n}}(f)\leq  \frac{1}{2} \sum_{x\in \Omega_n} \sum_{i=1}^n  \left[ \log  f(x) - \log f(\sigma_i(x))  \right]_{+}^2f(x)\mu^{\otimes n}(x) 
\end{eqnarray}
that holds for all product probability measures on the hypercube $\Omega_n=\{0,1\}^n$, for all dimensions $n\ge 1$. 

First we observe that, by tensorisation of the log-Sobolev inequality (see \textit{e.g.} \cite[Chapter 1]{ane}), we only need to prove 
Gross' Inequality \eqref{gross} in dimension one ($n=1$). Then, thanks to a result by Miclo \cite{miclo}, we know that extremal functions in the log-Sobolev inequality, in dimension one, are monotone. Hence, we can assume that $g$ is monotone and non-decreasing (the case $g$ non-increasing can be treated similarly). Furthermore, for convenience, we first assume that the function 
$g \colon \mathbb{R} \to \mathbb{R}$ is smooth and compactly supported.  

Let $\mu_p$ be the  Bernoulli probability measure with parameter $p\in [0,1]$. We apply \eqref{mod+logsob} to the function $f=e^{G_n}$, with 
$$
G_n(x)={ g \left( \frac{\sum_{i=1}^n x_i - np}{\sqrt{ np(1-p)}} \right)}, \qquad  x \in \Omega_n,
$$ 
so that $\ent_{\mu_p^{\otimes n}}\left(e^{G_n}\right)$ tends to $\ent_\gamma(e^g)$ by the Central Limit Theorem.
It remains to identify the limit, when $n$ tends to infinity, of the Dirichlet form (the first term in the right-hand side of \eqref{mod+logsob}). 
Let $\bar{x}^iy_i$ denote the vector $(x_1,\ldots,x_{i-1},y_i,x_{i+1},\ldots,x_{n})$.
Then,
 \begin{align*}
\sum_{x_i\in \{0,1\}} [G_n(x)-G_n(\sigma_i(x))]_+^2 e^{G_n(x)}\,\mu_p(x_i)&=p [G_n(\bar{x}^i1)-G_n(\bar{x}^i0)]_+^2 e^{G_n(\bar{x}^i1)}\\ &\quad
+ (1-p) [G_n(\bar{x}^i0)-G_n(\bar{x}^i1)]_+^2 e^{G_n(\bar{x}^i0)} .
\end{align*}
Now, since
\begin{align*}
\frac{\sum_{i=1}^n x_i - np}{\sqrt{ np(1-p)}} -  \frac{\sum_{j\neq i} x_j- (n-1)p}{\sqrt {(n-1)p(1-p)}}
&=\frac{x_i}{\sqrt{ np(1-p)}}+\frac 1{\sqrt{ p(1-p)}}\sum_{j\neq i} x_j\left(\frac 1{\sqrt{n}}-\frac 1{\sqrt{n-1}}\right) \\
& \quad +\frac p{\sqrt{ p(1-p)}}\left(\sqrt n-\sqrt{n-1}\right)\\
&=\frac{x_i}{\sqrt{ np(1-p)}}-\frac{\sum_{j\neq i} x_j}{ \sqrt{ p(1-p)}  \left({\sqrt{n}}+{\sqrt{n-1}}\right){\sqrt{n}}{\sqrt{n-1}}} \\
& \quad +\frac p{\sqrt{ p(1-p)}\left(\sqrt n+\sqrt{n-1}\right)}
=O\left(\frac1{\sqrt {n}}\right) ,
\end{align*}
 by a Taylor Expansion, we have
$$
G_n(\bar{x}^i1)-G_n(\bar{x}^i0)=\frac1{\sqrt{ np(1-p)}}\,g'\left(\frac{\sum_{j\neq i} x_j-p ({n-1})}{\sqrt {(n-1)p(1-p)}}\right)+O\left(\frac 1n \right).
$$
Setting $\displaystyle y_i(x)=\frac{\sum_{j\neq i} x_j-p( {n-1})}{\sqrt {(n-1)p(1-p)}}$, it  follows that 
$$
\sum_{x_i\in \{0,1\}}[G_n(x)-G_n(\sigma_i(x))]_+^2 e^{G_n(x)}\,\mu_p(x_i)
=
\frac { g'\left(y_i(x)\right)^2 e^{g(y_i(x))}}{n(1-p)}  + O\left(\frac1{n^{3/2}}\right).
$$
Now, since all $y_i(x)$'s have the same law under $\mu_p^{\otimes n}$, it follows that 
$$
 \sum_{x\in \Omega_n} \sum_{i=1}^n[G_n(x)-G_n(\sigma_i(x))]_+^2 e^{G_n(x)}\,\mu_p^{\otimes n}(x)
  =  \!\! \sum_{x\in \Omega_n} \! \frac{g'\left(y_1(x) \right)^2e^{g(y_1(x))}}{1-p} \mu_p^{ \otimes n}(x)  +O\left(\frac1{\sqrt {n}}\right).
$$
The desired result follows by the Central Limit Theorem, then optimizing over all $p\in [0,1]$, and finally by a standard density argument. This ends the proof.
\end{proof}

%%%%%%%%%%%%%%%%%%%%%%%%%%%%%%%%%%%%%%%%%%%%%%%%%%%%%%%%%%%%%%%%%%%%%%%%%%%%%%%%%%%%

\subsection{Complete graph}

Combining the differentiation property \eqref{crepe} together with the displacement convexity on the complete graph of Proposition  
\ref{convexity}, we shall prove the following result.

\begin{prop}[HWI type inequality on the complete graph] \label{hwicomplete}
Let $\mu \equiv 1/n$ be the uni\-form measure on the complete graph $K_n$. Then, for any 
$f \colon V(K_n) \to (0,\infty)$ with $\int fd\mu=1$, it holds
$$
\ent_\mu (f) \leq \mathcal{E}_\mu(f,\log f) - \frac{1}{2} \left( \Tw_2 (\mu|f \mu) + \Tw_2 (f \mu|\mu) \right)\,,
$$
where 
$$
\mathcal{E}_\mu(f,\log f) := \frac{1}{2} \sum_{x,y \in K_n} (f(y)-f(x))(\log f(y) - \log f(x))\mu(x)\mu(y) 
$$
corresponds to the Dirichlet form associated to the Markov chain on $K_n$ that jumps uniformly at random from any vertex to any vertex (\textit{i.e.}\ with transition probabilities $K(x,y)=\mu(y)=1/n$, for any $x$, $y \in V(K_n)$). 
\end{prop}

\begin{proof} We follow the same line of proof as in Proposition \ref{HWIprop}.
Fix $f \colon V(K_n) \to (0,\infty)$ with $\int fd\mu=1$.
By Proposition \ref{convexity}, applied to $\nu_1 = \mu$ (which implies that $H(\nu_1|\mu)=0$) and $\nu_0= f\mu$, we have
\begin{align*}
H(\nu_t  | \mu) 
 \leq 
(1-t) H(\nu_0|\mu) - \frac{t(1-t)}{2} \left( \Tw_2 (\nu_1|\nu_0) + \Tw_2 (\nu_0|\nu_1) \right) 
\end{align*}
where $\nu_t=(1-t)\nu_0 + t\nu_1$. 
Hence, as $t$ goes to 0, we get 
\begin{eqnarray*}
\int f \log f d\mu=H(\nu_0|\mu) \leq-\frac{\partial}{\partial t} H(\nu_t  | \mu)_{|_{t=0}} 
 - \frac{1}{2} \left( \Tw_2 (\nu_1|\nu_0) + \Tw_2 (\nu_0|\nu_1) \right). \\
\end{eqnarray*}
The expected result follows from \eqref{crepe}. 
\end{proof}

In the case of the two-point space, one can deal with any Bernoulli measure (not only the uniform one as in the case of the complete graph).

\begin{prop}[HWI for the two-point space]\label{hwitwopoints}
Let $\mu$ be a Bernoulli-$p$, $p \in (0,1)$ measure on the two-point space $\Omega_1=\{0,1\}$. Then, for any 
$f \colon \Omega_1 \to (0,\infty)$ with $\mu(f)=1$, it holds
$$
\ent_\mu(f) \leq \mathcal{E}_\mu(f,\log f) - \frac{1}{2}\left( \Tw_2 (\mu|f \mu) + \Tw_2 (f \mu|\mu) \right)
$$
where,  
$$
\mathcal{E}_\mu(f,\log f) = p(1-p)(f(1)-f(0))(\log f(1) - \log f(0)) .
$$
\end{prop}

\begin{proof} Reasoning as above, 
 Proposition \ref{convexity}, applied to $\nu_1 = \mu$ and $\nu_0= f\mu$, implies
$$
\ent_\mu (f) \leq - \frac{\partial}{\partial t} H(\nu_t  | \mu)_{|_{t=0}} - \frac{1}{2}\left( \Tw_2 (\mu|f \mu) + \Tw_2 (f \mu|\mu) \right) \,,
$$
where $\nu_t = (1-t)f\mu + t \mu$. Set $q=1-p$. 
Since $H(\nu_t |\mu) = [(1-t)f(0)q + t q] \log [(1-t)f(0)+t] + 
[(1-t)f(1)p + t p] \log [(1-t)f(1)+t]$, it immediately follows that
\begin{align*}
\frac{\partial}{\partial t} H(\nu_t  | \mu)_{|_{t=0}} 
& = 
q(1-f(0))\log f(0) + q(1-f(0)) %\\
%& \quad 
+ p(1-f(1))\log f(1) + p(1-f(1)) \\
& =
q(1-f(0))\log f(0)  + p(1-f(1))\log f(1)
\end{align*}
where the second equality follows from the fact that $p+q=1=\mu(f)=qf(0)+pf(1)$.
Using again that $1=qf(0)+pf(1)$, we observe that
$$
q(1-f(0))\log f(0) = pq(f(1) - f(0))\log f(0) 
$$
and
$$
p(1-f(1))\log f(1) = -pq(f(1)-f(0) \log f(1)\,,
$$
from which the expected result follows.
\end{proof}

%%%%%%%%%%%%%%%%%%%%%%%%%%%%%%%%%%%%%%%%%%%%%%%%%%%%%%%%%%%%%%%%%%%%%%%%%%%%%%%%%%%%
%%%%%%%%%%%%%%%%%%%%%%%%%%%%%%%%%%%%%%%%%%%%%%%%%%%%%%%%%%%%%%%%%%%%%%%%%%%%%%%%%%%%
%%%%%%%%%%%%%%%%%%%%%%%%%%%%%%%%%%%%%%%%%%%%%%%%%%%%%%%%%%%%%%%%%%%%%%%%%%%%%%%%%%%%

\section{Prekopa-Leindler type inequality}

In this section we show by a duality argument that the displacement convexity property implies a discrete version of the Prekopa-Leindler inequality.  (This argument was originally done by J. Lehec \cite{lehec12} in the context of Brascamp-Lieb inequalities.) Then we show that this Prekopa-Leindler inequality allows to recover the discrete  modified  log-Sobolev inequality \eqref{mod+logsob} and a weak version of the transport entropy inequality of Remark~\ref{remarcable}.

Let us first recall the statement of the usual Prekopa-Leindler inequality.

\begin{thm}[Prekopa-Leindler \cite{prekopa2,prekopa1,leindler}]
Let $n\in \N^*$ and $t\in [0,1]$. For all triples $(f,g,h)$ of measurable functions on $\R^n$ such that
$$h((1-t)x+ty)\geq (1-t)f(x)+tg(y),\qquad \forall x,y\in\R^n,$$
it holds
$$\int e^{h(z)}\,dz \geq \left(\int e^{f(x)}\,dx\right)^{1-t}\left(\int e^{g(y)}\,dy\right)^t.$$
\end{thm}
Using the identity (with $\| \cdot \|$ denoting the Euclidean norm),
$$
\frac{1}{2}\|(1-t)x+ty\|_2^2=(1-t)\frac{\|x\|_2^2}{2}+t\frac{\|y\|_2^2}{2}-t(1-t)\frac{\|x-y\|_2^2}{2}, 
\qquad x,y \in \mathbb{R}^n,
$$
one can recast, without loss, the preceding result into an inequality for the Gaussian distribution.
\begin{thm}[Prekopa-Leindler: the Gaussian case]\label{prekopagauss}
Let $\gamma_n$ be the standard normal distribution on $\R^n$ and $t\in [0,1]$. For all triples $(f,g,h)$ of measurable functions on $\R^n$ such that
\begin{equation}\label{usualpl}
h((1-t)x+ty)\geq (1-t)f(x)+tg(y)-\frac{t(1-t)}{2}\|x-y\|_2^2,\qquad \forall x,y\in \R^n,
\end{equation}
it holds that
$$
\int e^{h(z)}\,\gamma_n(dz) \geq \left(\int e^{f(x)}\,\gamma_n(dx)\right)^{1-t}\left(\int e^{g(y)}\,\gamma_n(dy)\right)^t.
$$
\end{thm}

The next result shows that a discrete Prekopa-Leindler inequality can be derived from the displacement convexity property of the relative entropy. 
\begin{thm}[Prekopa-Leindler (discrete version)] \label{th:dpl}
Let $n\in \N^*$, $t\in [0,1]$ and $\mu\in \mathcal{P}(V^n)$. Suppose that $\mu$ verifies the following  property: for any $\nu_0, \nu_1 \in \mathcal{P}(V^n)$, there exists a coupling $\pi \in \Pi(\nu_0,\nu_1)$ such that
\begin{equation} \label{hyppl1}
H(\nu_t^\pi |\mu) \leq (1-t) H(\nu_0|\mu) + t H(\nu_1|\mu) - ct(1-t) I_2^{(n)}(\pi).
\end{equation}
If $(f,g,h)$ is a triple of functions on $V^n$ such that: $\forall x\in V^n$, $\forall m\in \mathcal{P}(V^n)$\,,
\begin{align} \label{hyppl2}
\iint h(z) \,\nu_t^{x,y}(dz) m(dy)  
& \geq 
(1-t)f(x) + t \int g (y)\,m(dy)  %\nonumber  \\
%& \quad 
-ct(1-t) \sum_{i=1}^n \left(\int d(x_i,y_i)\,m(dy)\right)^2,
\end{align}
then it holds
$$
\int e^{h(z)}\,\mu(dz) \geq \left( \int e^{f(x)}\,\mu(dx) \right)^{1-t} \left( \int e^{g(y)} \,\mu(dy) \right)^t  .
$$
\end{thm}

\begin{proof}
Let $n \in \mathbb{N}$, $f, g, h : V^n \mapsto \R $,  $\mu \in \mathcal{P}(V^n)$,  $t \in [0,1]$ and $c \in (0, \infty)$
satisfying the hypotheses of the theorem. Given $\nu_0, \nu_1 \in \mathcal{P}(V^n)$, let $\pi$ be such that \eqref{hyppl1} 
 holds and let $p$ be such that $\pi(x,y)=\nu_0(x)p(x,y)$, $x,y \in V^n$.

Then, integrate \eqref{hyppl2} in the variable $x$ with respect to $\nu_0$, with $m(y)=p(x,y)$, so that (recall \eqref{atlantabis})
\begin{align*} 
 \int h \,d\nu_t^{\pi}   
 \geq 
(1-t) \int f\,d\nu_0  + t \int  g\,d\nu_1   
 -ct(1-t) I_2^{(n)}(\pi) .
\end{align*}
Together with \eqref{hyppl1}, we end up with
\begin{align*} 
 \int h \,d\nu_t^{\pi}   - H(\nu_t^\pi |\mu)
 \geq 
(1-t) \left( \int f d\nu_0 - H(\nu_0|\mu) \right) + t \left( \int g\,d\nu_1   - H(\nu_1|\mu)\right) .
\end{align*}
The result follows by optimization, since by duality (for any $\alpha \colon V^n \mapsto \R $)\,,
$$
\sup_{m \in \mathcal{P}(V^n)} \left\{ \int \alpha\,dm - H(m|\mu) \right\} = \log \int e^\alpha\,d\mu .
$$
This ends the proof.
\end{proof}An immediate corollary is a Prekopa-Leindler inequality on the discrete hypercube.
\begin{cor}\label{PL-cube}
Let $\mu$ be a probability measure on $\{0,1\}$, $n\in \N^*$ and $t\in [0,1]$. For all triple $(f,g,h)$ verifying \eqref{hyppl2} with $c=1/2$, it holds
$$
\int e^{h(z)}\,\mu^{\otimes n}(dz) \geq \left( \int e^{f(x)}\,\mu^{\otimes n}(dx) \right)^{1-t} \left( \int e^{g(y)} \,\mu^{\otimes n}(dy) \right)^t  .
$$
\end{cor}

It is well known that  Talagrand's transport-entropy  inequality and the logarithmic Sobolev inequality for the Gaussian measure are both consequences of the Prekopa-Leindler inequality of Theorem~\ref{prekopagauss} \cite{bobkov-ledoux00}.  
%\textcolor{red}{Add the reference where it is proved that PL implies $T_2$}. 
Similarly the  discrete version of  Prekopa Leindler inequality implies the modified logarithmic Sobolev inequality induced by Corollary \ref{cortarte} and the   transport-entropy inequality associated with the distance $\Tw_2$ of Remark \ref{remarcable}.
\begin{cor}\label{rrrouf}
Assume that the following Prekopa-Leindler inequality holds: for all $t\in (0,1)$, for all  triples of functions $(f,g,h)$ on $V^n$ such that:  $\forall x\in V^n$, $\forall m\in \mathcal{P}(V^n)$\,,
\begin{align*} \label{hyppl2}
\iint h(z) \,\nu_t^{x,y}(dz) m(dy)  
& \geq 
(1-t)f(x) + t \int g (y)\,m(dy) % \nonumber  \\
%& \quad 
-ct(1-t) \sum_{i=1}^n \left(\int d(x_i,y_i)\,m(dy)\right)^2,
\end{align*}
 it holds that
$$
\int e^{h(z)}\,\mu(dz) \geq \left( \int e^{f(x)}\,\mu(dx) \right)^{1-t} \left( \int e^{g(y)} \,\mu(dy) \right)^t  .
$$
Then one has, for all functions $h \colon V^n \to \mathbb{R}$,
\begin{eqnarray*}
\ent_\mu(e^h)\leq \frac{1}{4c} \sum_{x\in V^n} \sum_{i=1}^n \left[\sum_{z\in N_i(x)} \left(h(x) - h(z)  \right)\right]_{+}^2e^{h(x)}\mu(x).
\end{eqnarray*}
and for all probability measures $\nu$, absolutly continous with respect to $\mu$,
\begin{eqnarray}\label{transporttilde1}
c\,\Tw_2(\mu|\nu)\leq H(\nu|\mu),
\end{eqnarray}
\begin{eqnarray}\label{transporttilde2}
c\,\Tw_2(\nu|\mu)\leq H(\nu|\mu),
\end{eqnarray}
\end{cor}

\begin{proof} 
We first prove  the transport-entropy inequalities \eqref{transporttilde1} and \eqref{transporttilde2}.
Let $k$ be a function on $V^n$ (necessarily bounded, since $V$ is finite). We apply the discrete  Prekopa-Leindler inequality with $h=0$, $g=-(1-t)k$ and $f=tQk$, with $Qk$ defined so that the condition \eqref{hyppl2} holds: for all $x\in V^n$,
$$Qk(x)=\inf_{m\in \mathcal{P}(V^n) } \left\{  \int k ( y)\,m(dy)  
  + c 
  \sum_{i=1}^n \left(\int d(x_i,y_i)\,m(dy)  \right)^2     \right\} .$$
  Therefore,  one has for all $t\in (0,1)$,
 $$\left(\int e^{tQk}d\mu\right)^{1/t}\left(\int e^{-(1-t)k}\,d\mu\right)^{1/(1-t)}\leq 1.$$
  As $t$ goes to 1, we get for all functions $k$ on $V^n$,
  $$\int e^{Qk}d\mu\leq e^{\mu(k)},$$
  and this is known to be  a dual form of the transport-entropy inequality \eqref{transporttilde1} (see \cite{GRST}).
 Similarly as $t$ goes to 0, we get for all functions $k$ on $V^n$,
  $$\int e^{-k}d\mu\leq e^{-\mu(Qk)},$$
 which is   a dual form of the transport-entropy inequality \eqref{transporttilde2}.

Let us now turn to the proof of the modified discrete logarithmic Sobolev inequality.
Fix a bounded function $h:V^n\to \R$ and choose $g=th$ and $f=h+tR_th$ with $R_t h $ designed  so that condition \eqref{hyppl2} holds.  Namely, for all $x\in V^n$, 
\begin{align*}
R_t h(x)=\inf_m& \left\{  \frac{1}{t(1-t)}\left( \iint h(z)\nu_t^{x,y}(dz)\,m(dy) - (1-t) h(x)  \right)\right.\\
& \quad-\left.   \frac{t}{1-t} \int h ( y)\,m(dy)  
  + c \sum_{i=1}^n \left(\int d(x_i,y_i)\,m(dy)  \right)^2     \right\} \,,
\end{align*}
where the infimum runs over all probability measures $m \in \mathcal{P}(V^n)$.
Then the Prekopa-Leindler inequality reads
$$\int e^h d\mu \geq \left( \int e^h e^{t R_t h} d\mu \right)^{1-t} \left( \int e^{th} d\mu \right)^t,$$
which can be rewritten as
\begin{equation*} 
1 \geq \left( \int e^{t R_t h} d\mu_h \right)^{1/t} \left( \int e^{(t-1)h} d\mu_h\right)^{1/(1-t)},
\end{equation*}
with $d\mu_h=\frac{e^h}{\int e^h\,d\mu}\,d\mu.$
Letting $t$ go to $0$, we easily deduce (leaving some details to the reader) that,
$$
\int (\liminf_{t \to 0} R_t h) e^h d\mu \leq \int e^h d\mu \log \int e^h d\mu\,.
$$
This can  equivalently  be written as
$$
\ent_\mu (e^h) \leq \int (h - \liminf_{t \to 0} R_t h)e^h d\mu.
$$
We conclude using the following claim.
\begin{claim} \label{claim0}
For all $x \in \R $, we have
$$
h(x)-\liminf_{t \to 0} R_t h(x) \leq \frac{1}{4c}  \sum_{i=1}^n \left[\sum_{z\in N_i(x)} \left(h(x) - h(z)  \right)\right]_+^2.
$$
\end{claim}
\end{proof}

\begin{proof}[Proof of Claim \ref{claim0}]
By a Taylor expansion and by Proposition \ref{paris}, for all $x,y\in V^n$\,,
\begin{align*}
\int h(z)\nu_t^{x,y}(dz)
&= 
\nu_t^{x,y}(h) = \nu_0^{x,y}(h)+t d(x,y) \nu_0^{x,y}\left(\nabla^{x,y} h\right) + o(t)
=h(x)+t d(x,y) \nabla^{x,y} h(x)+ o(t),
\end{align*}
with the quantity $o(t)$ independent of $y$ since $h$ is bounded. Now, from the definition of the sets $N_i(x)$, $i\in \{1,\ldots,n\}$ and using the identity \eqref{changement}, one has
\begin{align*}
\nabla^{x,y} h(x)&=\frac1{|\Gamma(x,y)|} \sum_{\gamma\in\Gamma(x,y)}\left(h(\gamma_+(x))-h(x)\right)
= \sum_{z\in V_n, z\sim x} \left(h(z)-h(x)\right)\frac{|\Gamma(x,z,y)|}{|\Gamma(x,y)|}\\
&= \sum_{i=1}^n \sum_{z\in N_i(x)} \left(h(z)-h(x)\right)\frac{d(x_i,y_i)|\Gamma(x_i,z_i,y_i)|}{d(x,y)|\Gamma(x_i,y_i)|}.
\end{align*}
Therefore 
\begin{align*}
& h(x)- R_t h(x) =\sup_m  \left\{  \int \sum_{i=1}^n \sum_{z\in N_i(x)} \left(h(x)-h(z)\right)d(x_i,y_i)\frac{|\Gamma(x_i,z_i,y_i)|}{|\Gamma(x_i,y_i)|} \,m(dy)\right.\\
&\phantom{AAAAAAAAAAAAAA} -\left.     c 
  \sum_{i=1}^n \left(\int d(x_i,y_i)\,m(dy)  \right)^2     \right\}  +o(1)\\
  &\leq \sum_{i=1}^n \sup_m  \left\{  \left[\sum_{z\in N_i(x)} \left(h(x)-h(z)\right)\right]_+  \int  d(x_i,y_i) m(dy) - c\left(\int d(x_i,y_i)\,m(dy)  \right)^2     \right\}  +o(1)\\
  &\leq 
  \sum_{i=1}^n \sup_{v\geq 0}  \left\{ v \left[\sum_{z\in N_i(x)} \left(h(x)-h(z)\right) \right]_+ - cv^2  \right\} +o(1)
  =
  \frac 1{4c}\sum_{i=1}^n \left[\sum_{z\in N_i(x)} \left(h(x)-h(z)\right) \right]_+^2+o(1).
\end{align*}
The claim follows by letting $t$ go to 0.
\end{proof}

\subsection*{Acknowledgements} 
The authors thank Erwan Hillion for informing them of his independent work \cite{hillion} and his thesis \cite{hillionthesis}.
The French authors thank the hospitality of Georgia Institute of Technology, Atlanta, Georgia.
The last author is grateful to his French collaborators for bringing him back to this topic, and to the
Universit\'e Paris Est Marne la Vall\'ee - Laboratoire d'Analyse et de Math\'e\-matiques Appliqu\'ees -- 
for their generosity and hospitality  in hosting him.

%%%%%%%%%%%%%%%%%%%%%%%%%%%%%%%%%%%%%%%%%%%%%%%%
%%%%%%%%%%%%%%%%%%%%%%%%%%%%%%%%%%%%%%%%%%%%%%%%
%%%%%%%%%%%%%% Bibliography %%%%%%%%%%%%%%%%%%%%

\bibliographystyle{plain}
\bibliography{displacement}

\begin{thebibliography}{10}

\bibitem{ane}
C.~An{\'e}, S.~Blach{\`e}re, D.~Chafa{\"{\i}}, P.~Foug{\`e}res, I.~Gentil,
  F.~Malrieu, C.~Roberto, and G.~Scheffer.
\newblock {\em Sur les in\'egalit\'es de {S}obolev logarithmiques}, volume~10
  of {\em Panoramas et Synth\`eses [Panoramas and Syntheses]}.
\newblock Soci\'et\'e Math\'ematique de France, Paris, 2000.
\newblock With a preface by Dominique Bakry and Michel Ledoux.

\bibitem{bakry}
D.~Bakry.
\newblock L'hypercontractivit\'e et son utilisation en th\'eorie des
  semigroupes.
\newblock In {\em Lectures on probability theory ({S}aint-{F}lour, 1992)},
  volume 1581 of {\em Lecture Notes in Math.}, pages 1--114. Springer, Berlin,
  1994.

\bibitem{bobkov-ledoux-98}
S.~G. Bobkov and M.~Ledoux.
\newblock On modified logarithmic {S}obolev inequalities for {B}ernoulli and
  {P}oisson measures.
\newblock {\em J. Funct. Anal.}, 156(2):347--365, 1998.

\bibitem{bobkov-ledoux00}
S.~G. Bobkov and M.~Ledoux.
\newblock From {B}runn-{M}inkowski to {B}rascamp-{L}ieb and to logarithmic
  {S}obolev inequalities.
\newblock {\em Geom. Funct. Anal.}, 10(5):1028--1052, 2000.

\bibitem{bobkov-tetali}
S.G. Bobkov and P.~Tetali.
\newblock Modified logarithmic {S}obolev inequalities in discrete settings.
\newblock {\em J. Theoret. Probab.}, 19(2):289--336, 2006.

\bibitem{bonciocat}
A.I. Bonciocat and K.T. Sturm.
\newblock Mass transportation and rough curvature bounds for discrete spaces.
\newblock {\em J. Funct. Anal.}, 256(9):2944--2966, 2009.

\bibitem{PPP}
P.~Caputo, P.~Dai~Pra, and G.~Posta.
\newblock Convex entropy decay via the {B}ochner-{B}akry-{E}mery approach.
\newblock {\em Ann. Inst. Henri Poincar\'e Probab. Stat.}, 45(3):734--753,
  2009.

\bibitem{cordero-ms}
D.~Cordero-Erausquin, R.~J. McCann, and M.~Schmuckenschl{\"a}ger.
\newblock A {R}iemannian interpolation inequality \`a la {B}orell, {B}rascamp
  and {L}ieb.
\newblock {\em Invent. Math.}, 146(2):219--257, 2001.

\bibitem{csiszar}
I.~Csisz{\'a}r.
\newblock Information-type measures of difference of probability distributions
  and indirect observations.
\newblock {\em Studia Sci. Math. Hungar.}, 2:299--318, 1967.

\bibitem{DPPP}
P.~Dai~Pra, A.~M. Paganoni, and G.~Posta.
\newblock Entropy inequalities for unbounded spin systems.
\newblock {\em Ann. Probab.}, 30(4):1959--1976, 2002.

\bibitem{dembo97}
A.~Dembo.
\newblock Information inequalities and concentration of measure.
\newblock {\em Ann. Probab.}, 25(2):927--939, 1997.

\bibitem{erbar-maas}
M.~Erbar and J.~Maas.
\newblock Ricci curvature of finite {M}arkov chains via convexity of the
  entropy.
\newblock Preprint Available at arXiv:1111.2687, 2012.

\bibitem{GQ03}
F.~Gao and J.~Quastel.
\newblock Exponential decay of entropy in the random transposition and
  {B}ernoulli-{L}aplace models.
\newblock {\em Ann. Appl. Probab.}, 13(4):1591--1600, 2003.

\bibitem{goel}
S.~Goel.
\newblock Modified logarithmic {S}obolev inequalities for some models of random
  walk.
\newblock {\em Stochastic Process. Appl.}, 114(1):51--79, 2004.

\bibitem{GRST}
N.~Gozlan, C.~Roberto, P.M. Samson, and P.~Tetali.
\newblock Transport-entropy inequalities in discrete settings.
\newblock In preparation, 2012.

\bibitem{gross}
L.~Gross.
\newblock Logarithmic {S}obolev inequalities.
\newblock {\em Amer. J. Math.}, 97(4):1061--1083, 1975.

\bibitem{hillionthesis}
E.~Hillion.
\newblock Analyse et g\'eom\'etrie dans les espaces m{\'e}triques mesur{\'e}s :
  in{\'e}galit\'es de {B}orell-{B}rascamp-{L}ieb et conjecture de
  {O}lkin-{S}hepp.
\newblock PhD thesis, 2010.

\bibitem{hillion}
E.~Hillion.
\newblock Contraction of measures on graphs.
\newblock Preprint, 2012.

\bibitem{hillion-johnson}
E.~Hillion, O.~Johnson, and Y.~Yu.
\newblock Translation of probability measures on {Z}.
\newblock Preprint, 2010.

\bibitem{johnson}
O.~Johnson.
\newblock Log-concavity and the maximum entropy property of the {P}oisson
  distribution.
\newblock {\em Stochastic Process. Appl.}, 117(6):791--802, 2007.

\bibitem{joulin}
A.~Joulin.
\newblock A new {P}oisson-type deviation inequality for {M}arkov jump processes
  with positive {W}asserstein curvature.
\newblock {\em Bernoulli}, 15(2):532--549, 2009.

\bibitem{knothe}
H.~Knothe.
\newblock Contributions to the theory of convex bodies.
\newblock {\em Michigan Math. J.}, 4:39--52, 1957.

\bibitem{kullback}
S.~Kullback.
\newblock Lower bound for discrimination information in terms of variation.
\newblock {\em IEEE Trans. Information Theory}, 4:126–127, 1967.

\bibitem{lehec}
J.~Lehec.
\newblock Private communication.
\newblock 2012.

\bibitem{lehec12}
J.~Lehec.
\newblock Representation formula for the entropy and functional inequalities.
\newblock To appear in Annales de l'IHP; available at
  http://arxiv.org/abs/1006.3028., 2012.

\bibitem{leindler}
L.~Leindler.
\newblock On a certain converse of {H}\"older's inequality.
\newblock In {\em Linear operators and approximation ({P}roc. {C}onf.,
  {O}berwolfach, 1971)}, pages 182--184. Internat. Ser. Numer. Math., Vol. 20.
  Birkh\"auser, Basel, 1972.

\bibitem{leonard}
C.~L\'eonard.
\newblock Private communication.
\newblock 2012.

\bibitem{lin-yau}
Y.~Lin and S.-T. Yau.
\newblock Ricci curvature and eigenvalue estimate on locally finite graphs.
\newblock {\em Math. Res. Lett.}, 17(2):343--356, 2010.

\bibitem{LV09}
J.~Lott and C.~Villani.
\newblock Ricci curvature for metric-measure spaces via optimal transport.
\newblock {\em Ann. of Math. (2)}, 169(3):903--991, 2009.

\bibitem{maas}
J.~Maas.
\newblock Gradient flows of the entropy for finite {M}arkov chains.
\newblock {\em J. Funct. Anal.}, 261(8):2250--2292, 2011.

\bibitem{marton96}
K.~Marton.
\newblock A measure concentration inequality for contracting {M}arkov chains.
\newblock {\em Geom. Funct. Anal.}, 6(3):556--571, 1996.

\bibitem{marton97}
K.~Marton.
\newblock Erratum to: ``{A} measure concentration inequality for contracting
  {M}arkov chains'' [{G}eom.\ {F}unct.\ {A}nal.\ {\bf 6} (1996), no. 3,
  556--571; {MR}1392329 (97g:60082)].
\newblock {\em Geom. Funct. Anal.}, 7(3):609--613, 1997.

\bibitem{marton99}
K.~Marton.
\newblock On a measure concentration inequality of {T}alagrand for dependent
  random variables.
\newblock {\em Preprint}, 1999.

\bibitem{mccann}
R.~J. McCann.
\newblock A convexity principle for interacting gases.
\newblock {\em Adv. Math.}, 128(1):153--179, 1997.

\bibitem{miclo}
L.~Miclo.
\newblock Monotonicity of the extremal functions for one-dimensional
  inequalities of logarithmic {S}obolev type.
\newblock In {\em S\'eminaire de {P}robabilit\'es {XLII}}, volume 1979 of {\em
  Lecture Notes in Math.}, pages 103--130. Springer, Berlin, 2009.

\bibitem{mielke}
A.~Mielke.
\newblock Geodesic convexity of the relative entropy in reversible {M}arkov
  chains.
\newblock To appear in Calc. Var. Part. Diff. Equ., 2012.

\bibitem{ollivier}
Y.~Ollivier.
\newblock Ricci curvature of {M}arkov chains on metric spaces.
\newblock {\em J. Funct. Anal.}, 256(3):810--864, 2009.

\bibitem{ollivier-villani}
Y.~Ollivier and C.~Villani.
\newblock A curved {B}runn--{M}inkowski inequality on the discrete hypercube;
  {O}r: {W}hat is the {R}icci curvature of the discrete hypercube?
\newblock To appear in {SIAM J}. {D}iscrete {M}ath., 2012.

\bibitem{OV00}
F.~Otto and C.~Villani.
\newblock Generalization of an inequality by {T}alagrand and links with the
  logarithmic {S}obolev inequality.
\newblock {\em J. Funct. Anal.}, 173(2):361--400, 2000.

\bibitem{pinsker}
M.~S. Pinsker.
\newblock {\em Information and information stability of random variables and
  processes}.
\newblock Translated and edited by Amiel Feinstein. Holden-Day Inc., San
  Francisco, Calif., 1964.

\bibitem{prekopa2}
A.~Pr{\'e}kopa.
\newblock Logarithmic concave measures with application to stochastic
  programming.
\newblock {\em Acta Sci. Math. (Szeged)}, 32:301--316, 1971.

\bibitem{prekopa1}
A.~Pr{\'e}kopa.
\newblock On logarithmic concave measures and functions.
\newblock {\em Acta Sci. Math. (Szeged)}, 34:335--343, 1973.

\bibitem{sturm-vonrenesse}
M-K.~von Renesse and K.T. Sturm.
\newblock Transport inequalities, gradient estimates, entropy, and {R}icci
  curvature.
\newblock {\em Comm. Pure Appl. Math.}, 58(7):923--940, 2005.

\bibitem{roberto-thesis}
C.~Roberto.
\newblock In\'egalit\'es de {H}ardy et de {S}obolev logarithmiques.
\newblock PhD thesis, 2001.

\bibitem{rosenblatt}
M.~Rosenblatt.
\newblock Remarks on a multivariate transformation.
\newblock {\em Ann. Math. Statistics}, 23:470--472, 1952.

\bibitem{Saloff}
L.~Saloff-Coste.
\newblock Lectures on finite {M}arkov chains.
\newblock In {\em Lectures on probability theory and statistics (Saint-Flour,
  1996)}, pages 301--413. Springer, Berlin, 1997.

\bibitem{sammer-tetali}
M.~Sammer and P.~Tetali.
\newblock Concentration on the discrete torus using transportation.
\newblock {\em Combin. Probab. Comput.}, 18(5):835--860, 2009.

\bibitem{sammer-thesis}
M.~D. Sammer.
\newblock {\em Aspects of mass transportation in discrete concentration
  inequalities}.
\newblock ProQuest LLC, Ann Arbor, MI, 2005.
\newblock Thesis (Ph.D.)--Georgia Institute of Technology.

\bibitem{S00}
P.M. Samson.
\newblock Concentration of measure inequalities for {M}arkov chains and
  {$\Phi$}-mixing processes.
\newblock {\em Ann. Probab.}, 28(1):416--461, 2000.

\bibitem{samson}
P.M. Samson.
\newblock Infimum-convolution description of concentration properties of
  product probability measures, with applications.
\newblock {\em Ann. Inst. H. Poincar\'e Probab. Statist.}, 43(3):321--338,
  2007.

\bibitem{sturm05}
K.T. Sturm.
\newblock Convex functionals of probability measures and nonlinear diffusions
  on manifolds.
\newblock {\em J. Math. Pures Appl. (9)}, 84(2):149--168, 2005.

\bibitem{St06a}
K.T. Sturm.
\newblock On the geometry of metric measure spaces. {I}.
\newblock {\em Acta Math.}, 196(1):65--131, 2006.

\bibitem{St06b}
K.T. Sturm.
\newblock On the geometry of metric measure spaces. {II}.
\newblock {\em Acta Math.}, 196(1):133--177, 2006.

\bibitem{villani03}
C.~Villani.
\newblock {\em Topics in optimal transportation}, volume~58 of {\em Graduate
  Studies in Mathematics}.
\newblock American Mathematical Society, Providence, RI, 2003.

\bibitem{villani}
C.~Villani.
\newblock {\em Optimal transport}, volume 338 of {\em Grundlehren der
  Mathematischen Wissenschaften [Fundamental Principles of Mathematical
  Sciences]}.
\newblock Springer-Verlag, Berlin, 2009.
\newblock Old and new.

\end{thebibliography}

\end{document}